\documentclass[a4paper,10pt]{amsart}

\usepackage{amsthm}
\usepackage{amssymb}
\usepackage{amsmath,stmaryrd}
\usepackage{mathtools}
\usepackage{mathrsfs}
\usepackage{pdfpages}
\usepackage{changes}
\usepackage{exscale,cite,color,amsopn}
\usepackage[colorlinks=true,pdftex,unicode=true,linktocpage,bookmarksopen,hypertexnames=true]{hyperref}
\usepackage{aliascnt}
\usepackage{colortbl}
\usepackage{enumitem}
\usepackage{verbatim}
\usepackage{hyperref}
\usepackage{tikz-cd}
\usepackage[capitalise]{cleveref}

\usetikzlibrary{trees}

\DeclareMathOperator{\id}{id}

\DeclareMathOperator{\dset}{\downarrow}
\DeclareMathOperator{\uset}{\uparrow}

\newcommand{\N}{\mathbb{N}}

\newcommand{\End}{\mathrm{End}}

\newcommand{\LH}{\mathrm{LH}}

\newcommand{\Spec}{\operatorname{Spec}}
\newcommand{\sym}{\mathrel{\mskip1mu\reflectbox{$\propto$}\mskip-1mu}}

\makeatletter
\numberwithin{equation}{section}
\numberwithin{figure}{section}
\numberwithin{table}{section}
\newtheorem{thm}{Theorem}[section]
\newtheorem*{thm*}{Theorem}
\newtheorem{lem}[thm]{Lemma}
\newtheorem{cor}[thm]{Corollary}
\newtheorem{pro}[thm]{Proposition}

\theoremstyle{definition}
\newtheorem{defn}[thm]{Definition}

\newtheorem{conjecture}[thm]{Conjecture}
\newtheorem*{convention*}{Convention}
\newtheorem{rem}[thm]{Remark}
\newtheorem{exa}[thm]{Example}

\makeatother

\makeatother
\title{L-algebras and their ideals: from simplicity to semidirect products}
\author{Silvia Properzi and Yufei Qin}
\address[Silvia Properzi]{Department of Mathematics and Data Science, Vrije Universiteit Brussel, Pleinlaan 2, 1050 Brussel, Belgium}
\email{Silvia.Properzi@vub.be}

\address[Yufei Qin]{Department of Mathematics and Data Science\\ Vrije Universiteit Brussel\\ Pleinlaan 2, 1050 Brussels, Belgium \\ School of Mathematical Sciences\\ Key Laboratory of Mathematics and Engineering Applications (Ministry of Education)\\ Shanghai Key laboratory of PMMP\\	East China Normal University\\	Shanghai 200241,	China}
 \email{Yufei.Qin@vub.be}
\begin{document}
\begin{abstract}
In this paper, we investigate the ideals of semidirect products of L-algebras and the structure of simple L-algebras. We provide a precise characterization of the ideals of semidirect products and describe the structure of their prime spectrum. 
Furthermore, we introduce a family of finite simple L-algebras and prove that 
every simple linear L-algebra belongs to this family.
We also show that the family we construct coincides with the class of simple algebras in a certain subclass of finite CKL-algebras.
As an application, we use these results to give a clear description of linear Hilbert algebras and their symmetric semidirect products.
\end{abstract}

\maketitle

\noindent \textbf{Keywords:} \textit{L-algebra, KL-algebra, CKL-algebra, Hilbert algebra, Semidirect product.}

\section*{Introduction}

Rump \cite{Rump_selfsimilar} introduced the notion of an L-algebra as a unifying algebraic framework arising from the study of the Yang--Baxter equation \cite{Rump_cyclesets,Rump_YBE,EtingofSchedlerSoloviev1999}, lattice-ordered groups \cite{AndersonFeil1988,Darnel_1995LatticeOrderedGroups}, and algebraic logics \cite{Rump_classiclogic}.  
For example, L-algebras generalize a wide range of logical algebraic structures, including Brouwerian semilattices \cite{Koehler1978Brouwerian}, MV-algebras \cite{Chang1958MVL,Chang1959Lukasiewicz,GispertMundici2005MV}, orthomodular lattices \cite{Loomis1955Dimension}, Hilbert algebras \cite{Diego66,Diego1965Hilbert,Horn_1962Separation}, and Glivenko algebras \cite{Rump_Glivenko}.  
Collectively, these structures encompass classical propositional logic, residuation theory, and the algebraic semantics of main non-classical logics \cite{Gottwald2005FuzzyLogic}.

In \cite{Rump_cyclesets}, it was shown that left non-degenerate involutive set-theoretic solutions of the Yang--Baxter equation correspond precisely to sets equipped with bijective left multiplications $\sigma_x : X \to X$, defined by $\sigma_x(y) = x \cdot y$, that satisfy the \emph{cycloid equation}
\[
  (x \cdot y) \cdot (x \cdot z) = (y \cdot x) \cdot (y \cdot z).
\]
An L-algebra is therefore a set $X$ endowed with a binary operation $(x,y)\mapsto x\cdot y$ satisfying the cycloid equation together with the axioms listed in \cref{Def: L-algebra}.  
Its \emph{structure group} $G(X)$ is defined as the quotient group of the self-similar closure $S(X)$ (see \cref{Def: self-similar closure}).

The structure group $G(X)$ naturally belongs to the class of  \emph{right $\ell$-groups} \cite{Rump_RightLGroups}, that are groups equipped with a lattice order invariant under right multiplication.  
This class contains, in particular, Artin’s braid groups \cite{BrieskornSaito_Artin,Deligne_Immeubles} and Garside groups \cite{DehornoyParis_Garside,Dehornoy_Garside}.  
Moreover, it was shown in \cite{Rump2015RightLGarside} that noetherian right $\ell$-groups with duality correspond precisely to the non-degenerate unitary set-theoretic solutions of the Yang--Baxter equation.

In recent years, substantial progress has been made on understanding ideals and structural properties of L-algebras.  
Rump and Vendramin \cite{RV_primes} proved that the lattice of ideals $\mathscr{I}(X)$ of an L-algebra $X$ is distributive and used this to determine the ideals and prime spectra of direct products of L-algebras.  
In \cite{DietzelMenchonVendramin}, the authors carried out an in-depth study of finite linear L-algebras and their isomorphism classes.

The purpose of this paper is to study the ideals of semidirect products of L-algebras and the structure of simple L-algebras.  
Our first main result is the following (see also \cref{Theorem:The equivalence of the ideal in semidirect product}):

\begin{thm*}\label{Theorem_inroduction:The equivalence of the ideal in semidirect product}
Let $X$ and $Y$ be L-algebras such that $Y$ operates on $X$ via $\rho$.  
Then $K$ is an ideal of $X \rtimes_{\rho} Y$ if and only if $\rho$ induces an operation 
\[
\tilde{\rho} : Y / K_Y \longrightarrow \End(X / K_X)
\]
such that
\[
(X \rtimes_{\rho} Y)\big/ (K_X \rtimes_{\rho|_{K_Y}} K_Y )
\;\cong\;
X / K_X \rtimes_{\tilde{\rho}} Y / K_Y .
\]

\end{thm*}

\medskip

The paper is organized as follows.

In \cref{Section: Preliminaries}, we recall the definitions and basic notions related to L-algebras.

In \cref{Section: Semidirect product and self-similarity}, we show that self-similar closure is compatible with semidirect products of L-algebras (see \cref{Thm: Gothere}).

In \cref{Section: Ideals of Semidirect Products of L-Algebras}, we study the ideals and prime ideals of semidirect products of L-algebras and of symmetric semidirect products of CKL-algebras.  
We prove \cref{Theorem:The equivalence of the ideal in semidirect product} and introduce the notion of $\rho$-prime ideals.  
We further show that the lattice of $\rho$-ideals $\rho\mathscr{I}(X)$ is distributive, and that the spectrum of $X\rtimes_{\rho} Y$ is naturally the disjoint union of the $\rho$-spectrum of $X$ and the spectrum of $Y$.  
As an application, we establish a natural correspondence between the ideals of the semidirect product and those of the symmetric semidirect product of CKL-algebras (see \cref{Thm: ideals of symmetric semidirect product}).

In \cref{Section: Simple linear L-algebras and CKL-algebras}, we introduce a family of simple CKL-algebras $\{\mathbf{A}_n\}_{n\geq 1}$ (see \cref{A_n is CKL}) and investigate linear L-algebras using ideal-theoretic methods.  
We prove that every simple linear L-algebra of size $n$ is isomorphic to $\mathbf{A}_n$ (see \cref{Thm: linearL_ideals}).  
We also introduce the notion of tail$^{+}$ L-algebras (see \cref{Def: tail+}), and show that every simple tail$^{+}$ CKL-algebra is linear (see \cref{Thm: tail+ L-algebra}).  
Furthermore, we conjecture that all simple CKL-algebras are linear.

In \cref{Section: Symmetric semidirect products and Hilbert algebras}, we investigate linear Hilbert algebras and their symmetric semidirect products in detail.

\section{Preliminaries}\label{Section: Preliminaries}
In this section, we collect the basic concepts and results on L-algebras that will be required in the sequel.

  \begin{defn}\label{Def: L-algebra}
      An \emph{L-algebra} is a set $X$ equipped with a binary operation $$(x,y)\mapsto x\cdot y,\qquad x,y\in X$$ and a distinguished element $1\in X$, satisfying the following conditions:
\begin{align}
        &1\cdot x=x,\quad x\cdot 1=x\cdot x=1  \label{logic unit}  \\
        &(x\cdot y)\cdot (x\cdot z)=(y\cdot x)\cdot (y\cdot z)\label{Cycle Sets eq}\\
    &    x\cdot y=y\cdot x=1 \Longrightarrow x=y.
    \end{align}
  \end{defn}

Several notable subclasses of L-algebras arise 
from additional identities.
Let $X$ be an L-algebra.
\begin{itemize}
\item  $X$ is a \emph{KL-algebra} if $x\cdot (y\cdot x)=1$,
      for $x,y\in  X$.
\item $X$ is a \emph{CKL-algebra} if 
$x\cdot(y\cdot z)=y\cdot(x\cdot z)$,
  for every $x, y, z\in X$.
\item $X$ is a \emph{Hilbert algebra} if $x\cdot(y\cdot z)=(x\cdot y)\cdot(x\cdot z)$, for every $x, y, z\in X$.
\end{itemize}
These classes are related by inclusion: every Hilbert algebra is CKL, and every CKL algebra is KL.

An L-algebra that possesses a smallest element $0$ is called a \emph{bounded L-algebra}.  
In this case, one can define the \emph{negation} by $x^{*}:= x \cdot 0$.
Bounded CKL-algebras are known as \emph{Glivenko algebras}.
Given an L-algebra $X$, for each element $x\in X$ we denote by $\sigma_x:X\to X$ the map $\sigma_x(y)=x\cdot y$.

Every L-algebra carries a natural partial order defined by $x\leq y$ if and only if $x\cdot y=1$.
For every element $x$ of an L-algebra $X$, we denote by $\dset x$
the downset $\{y\in X\mid y\leq x\}$ and by $\uset x$ the upset $\{y\in X\mid y\geq x\}$.
An element $x$ of an L-algebra $X$ is \emph{invariant} if $y\cdot x=x$ for all $y>x$,
while it is called \emph{prime} if $x\neq 1$ and $y\cdot x\leq x$ for every $y>x$.

\begin{defn}
    Let $X$ be an L-algebra and $S\subseteq X$.
    We say that $S$ is an L-\emph{subalgebra} if it is closed under the L-algebra operation of $X$.
\end{defn}

\begin{defn}
     We say that a subset $I\subseteq X$ is an \emph{ideal} of the L-algebra $X$ if it satisfies the following properties:
    \begin{enumerate}
        \item[(I1)] If $x\in I$ and $x\cdot y\in I$ then $y\in I$.
        \item[(I2)] If $x\in I$ then $y\cdot x\in I$ for all $y\in X$.
        \item[(I3)]  If $x\in I$ then $(x\cdot y)\cdot y\in I$ for all $y\in X$.
        \item[(I4)] If $x\in I$ then $y\cdot (x\cdot y)\in I$ for all $y\in X$.
    \end{enumerate}    
\end{defn}
Condition (I3) already implies that every ideal is closed under the L-algebra operation, so every ideal is an L-subalgebra. 
Simpler characterizations for ideals exist for special subclasses.

\begin{rem}\label{rem: ideals of KL and CKL-alg}
   Let $X$ be a KL-algebra.
    Then $I\subseteq X$ is an ideal of $X$ if and only if
    $I$ satisfies (I1) and (I3).   
      
    If $X$ is a CKL-algebra,
    then $I\subseteq X$ is an ideal of $X$ if and only if $1\in I$ and 
    $I$ satisfies (I1).
\end{rem}

If $X$ is an L-algebra, we denote by $\mathscr{I}(X)$ the set of ideals of $X$.
This set is itself an L-algebra 
with the binary operation defined as 
\[
I\cdot J=\{x\in X\mid \langle x \rangle \cap I\subset J\}.
\]
Note that $(I\cdot J)\cap I\subseteq J$ and for every ideal $K$ such that $K\cap I\subseteq J$ then $K\subseteq I\cdot J$.
Moreover, if $I\subseteq J$, then
$\langle x \rangle \cap I\subseteq J$,
for every $x\in I$. So
$I\cdot J=X$.

Using this structure, we can now define the notion of prime ideals. 
A prime ideal of $X$ is simply a prime element in the L-algebra of ideals $\mathscr{I}(X)$:
\begin{defn}
    A proper ideal $P$ of an L-algebra $X$ is \emph{prime}
    if for every ideal $I$ of $X$ either $I\subseteq P$ or $I\cdot P\subseteq P$.
\end{defn}
The prime ideals of $\mathscr{I}(X)$ form a topological space $\Spec(X)$, called the \emph{spectrum} of $X$, whose open sets are the collections $\{\mathscr{U}_I\}_{I\in \mathscr{I}(X)}$, where
\[
\mathscr{U}_I := \{\, P \in \Spec(X) \mid I \not\subseteq P \,\}.
\]

Furthermore, in \cite{RV_primes}, the following results are proven.
\begin{thm}\label{Thm: The elements in joint ideals}
    Let $I$ and $J$ be ideals of an L-algebra $X$. Then $y \in X$ belongs to $I \vee J$ if and only if there is an element $x \in I$ with $x \equiv y(\bmod J)$.
\end{thm}

\begin{thm}\label{Thm: The distributive rule of L-algebras}
The lattice of ideals of an L-algebra $X$ is distributive.
\end{thm}

In \cite{Rump_semidirect} they introduce the concept of \emph{semidirect product} of L-algebras that needs also the concept of \emph{operation}
of an L-algebra on another one.

\begin{defn}
    Let $X$ and $Y$ be L-algebras.
    $Y$ \emph{operates} on $X$ if there is a map $\rho:Y\to \End(X)$ such that
    \begin{enumerate}
        \item $\rho_1=\id$
        \item $\rho_{u\cdot v}\circ\rho_u=\rho_{v\cdot u}\circ\rho_v$ for all $u,v\in Y$.
    \end{enumerate}
\end{defn}
For example, an L-algebra $X$ operates on itself via
    $\rho_u(x)=u\cdot x$.
With this, we can now form semidirect products of L-algebras.
\begin{defn}
    Let $X$ and $Y$ be L-algebras such that $Y$ operates on $X$ via $\rho$.
    The \emph{semi-direct product} $X\rtimes_\rho Y$
    is the L-algebra defined on the set $X\times Y$,
    with operation
    \[
    (x,u)\cdot(y,v)=\big(\rho_{u\cdot v}(x)\cdot \rho_{v\cdot u}(y), u\cdot v\big).
    \]
\end{defn}

Note that the semidirect product of KL(CKL or Hilbert)-algebras is, in general, no longer a KL(CKL or Hilbert)-algebra. Therefore, we need to restrict the semidirect product in these cases.

\begin{thm}\label{Thm: semi-direct product of KL-alg}
  Let $X$ and $Y$ be L-algebras such that $Y$ operates on $X$ via $\rho$. 
  The semi-direct product $X\rtimes_\rho Y$ of L-algebras  is a KL-algebra if and only if
  $X$ and $Y$ are KL-algebras such that $x\cdot \rho_u(x)=1$ holds for $x\in X$ and $u\in Y$.
\end{thm}

\begin{defn}\label{Def: operate on KL-alg}
    Let $X$ and $Y$ be KL-algebras.
    $Y$ \emph{operates} on $X$ as KL-algebras if $Y$ operates on $X$ via $\rho$ as L-algebras and
    \[
    x\cdot \rho_u(x)=1,
    \] holds for $x\in X$ and $u\in Y$.
\end{defn}

\begin{defn}\label{Def: operate on CKL-alg}
    Let $X$ and $Y$ be CKL-algebras.
    $Y$ \emph{operates} on $X$ as CKL-algebras if $Y$ operates on $X$ via $\rho$ as KL-algebras and
    \begin{enumerate}
        \item $\rho_u\rho_v=\rho_v\rho_u$ for all $u,v\in Y$.\label{Def: operate on CKL-alg part 1}
        \item $\rho_u(x\cdot y)=x\cdot \rho_u(y)$ for all $u\in Y$ and $x\in X$.\label{Def: operate on CKL-alg part 2}
    \end{enumerate}
\end{defn}

\begin{defn}\label{def: symmetric semi-direct product}
    Let $X$ and $Y$ be CKL-algebras such that $Y$ operates on $X$ via $\rho$ as CKL-algebras.
    The \emph{symmetric semi-direct product} is
    the CKL algebra
    \[
    X\sym_\rho Y=\{(x,u)\in X\rtimes_\rho Y\mid \rho_u(x)=x\}.
    \]
\end{defn}

\begin{defn}
    Let $X$ and $Y$ be Hilbert algebras.
    $Y$ \emph{operates} on $X$ if there is a map $\rho:Y\to \End(X)$ such that
    $Y$ operates on $X$ as CKL-algebras and $\rho_u^2=\rho_u$ for all $u\in Y$.
\end{defn}

The example given before still works with the hypothesis of being Hilbert.
More precisely, any Hilbert algebra $X$ operates on itself as a Hilbert algebra through
 $\rho_u(x)=u\cdot x$.

\begin{defn}
    Let $X$ and $Y$ be Hilbert algebras such that $Y$ operates on $X$ via $\rho$.
    The \emph{symmetric semi-direct product} is
    the Hilbert algebra
    \[
    X\sym_\rho Y=\{(x,u)\in X\rtimes_\rho Y\mid \rho_u(x)=x\}.
    \]
\end{defn}

The concept of self-similarity is introduced in \cite{Rump_selfsimilar}
where it also proves the existence of the self-similar closure of any L-algebra.
\begin{defn}\label{Def: self-similar}
    An L-algebra $X$ is \emph{self-similar} if for every $x\in X$ the left multiplication $\sigma_x$
    induces a bijection between $\dset x$ and $X$.
\end{defn}

\begin{defn}\label{Def: self-similar closure}
    Let $X$ be an L-algebra,
    the \emph{self-similar closure} $S(X)$ of $X$ is
    a self-similar L-algebra with $X$ as an L-subalgebra which generates $S(X)$ as a monoid.
\end{defn}

The construction of the self-similar closure of an L-algebra given in \cite{Rump_Glivenko} uses the following theorems.
\begin{thm}\label{Thm: Extension for free monoid}
Let $(X, \cdot)$ be an L-algebra, and let $M(X)$ be the free monoid generated by $X$ with unit 1 and multiplication denoted by juxtaposition. 
Then the L-algebra operation of $X$ admits 
a unique extension to $M(X)$ such that 
\begin{align*}
& a b \cdot c=a \cdot(b \cdot c) \\
& a \cdot b c=((c \cdot a) \cdot b)(a \cdot c),\\
&1 \cdot a=a,
\end{align*}
for all $a, b, c \in M(X)$.
\end{thm}

\begin{thm}
     Let $X$ be an L-algebra. The self-similar closure of $X$ is defined as the quotient. 
     \[
     S(X)=M(X)/{\approx} ,
     \]
     where $a \approx b$ if and only if
     \[
     (c \cdot a) \cdot d=(c \cdot b) \cdot d,
     \]
     for all $c, d \in M(X)$.
\end{thm}

Moreover, L-algebra maps with codomain a self-similar one can be extended to the self-similar closure of the domain.
Hence, we have a functor $S$ that
is left adjoint to the inclusion of the category of self-similar L-algebras in the category of L-algebras.
\begin{pro}\label{Prop: universal property for S}
    Let $f: X \rightarrow H$ be a morphism of L-algebras, where $H$ is self-similar. Then $f$ has a unique extension to a morphism 
    $S(f): S(X) \rightarrow H$ of L-algebras. Moreover, every such extension $S(f)$ is multiplicative.
\end{pro} 

\begin{pro}\label{Prop: Self-silimar for KL-algebras}
    Let $X$ be a KL-algebra. Then $S(X)$ is a KL-algebra.
\end{pro}

However, if $X$ is a CKL-algebra, $S(X)$ is not necessarily a CKL-algebra, as shown by the following counterexample.

\begin{exa}\label{Exa: the Self-similar of CKL-algebra is not CKL-algebra}
    Let $X=\{1,x,y\}$ with $x\cdot y=y\cdot x=x$.
    It is easy to prove that $X$ is a CKL-algebra.
    But, on the other hand, $S(X)$ is not CKL as $x\cdot(y\cdot x^2)\neq y\cdot(x\cdot x^2)$:
    \begin{align*}
        x\cdot(y\cdot x^2)&=
        x\cdot \big(((x\cdot y)\cdot x)(y\cdot x)\big)=
        x\cdot \big((x\cdot x)x\big)=
        x\cdot x=
        1;\\
        y\cdot(x\cdot x^2)&=
        y\cdot \big(((x\cdot x)\cdot x)(x\cdot x)\big)=
        y\cdot \big((1\cdot x)1\big)=
        y\cdot x=
        x.
    \end{align*}
\end{exa}

\section{Semidirect product and self-similarity}\label{Section: Semidirect product and self-similarity}
In this section, we investigate the interplay between semidirect products and self-similarity. 

The following lemma describes the downsets of elements of semidirect products.

\begin{lem}
\label{X and Y inside semidirect}
    Let $X$ and $Y$ be L-algebras such that $Y$ operates on $X$ via $\rho$.
    Then for every $(x,u)\in X\rtimes_\rho Y$ we have that
    \begin{enumerate}
    \item $\dset (x,u)=\{(y,v)\in X\rtimes_\rho Y\mid v\in\dset u \text{ and } y\in \dset \rho_{u\cdot v}(x)\}$.
    \label{part1}
        \item $\dset(1,u)=X\times \dset u$ and $\sigma_{(1,u)}(y,v)=(y,u\cdot v)$,
        for every $(y,v)\in \dset(1,u)$.
        \label{part2}
        \item $\dset(x,1)=\bigcup_{v\in Y}\dset \rho_v(x)\times \{v\}$ and $\sigma_{(x,1)}(y,v)=(\rho_v(x)\cdot y,v)$,
        for every $(y,v)\in \dset(x,1)$.
        \label{part3}
        \item $(1,u)\cdot(x,u)=(x,1)$ and $X\rtimes_\rho Y=\{(x,1)(1,u)~|~x\in X,u\in Y\}$.
        \label{part4}
    \end{enumerate}
\end{lem}
\begin{proof}~
\begin{enumerate}[font=\itshape]
    \item Let $(x,u), (y,v)\in X\rtimes_\rho Y$.
    Then $(y,v)\leq (x,u)$ if and only if
    \[
    (1,1)=(y,v)\cdot(x,u)=\big(\rho_{v\cdot u}(y)\cdot \rho_{u\cdot v}(x), v\cdot u\big).
    \]
    The last condition is equivalent to $v\cdot u=1$ and $\rho_{v\cdot u}(y)\cdot \rho_{u\cdot v}(x)=1$,
    i.e. $v\cdot u=1$ and $y\cdot \rho_{u\cdot v}(x)=1$.
    \item  Let $(x,v)\in X\rtimes_\rho Y$. Then $(x,v)\leq (1,u)$ if and only if 
    $$ (1,1)=(x,v)\cdot(1,u)=(\rho_{v\cdot u}(x)\cdot\rho_{u\cdot v}(1),v\cdot u)=(1,v\cdot u).$$
    Thus $\dset(1,u)=X\times \dset u$. 
    For every $(y,v)\leq(1,u)$, we have that $v\cdot u=1$ and $\rho_{u\cdot v}(1)=1$. 
    Thus, $\sigma_{(1,u)}(y,v)=(y,u\cdot v)$.
    \item  Let every $(y,v)\in X\rtimes_\rho Y $.
    Then $(y,v)\leq (x,1)$  if and only if $y\cdot\rho_v(x)=1$.
    \item The result is directly calculated.\qedhere
\end{enumerate}
\end{proof}

Thanks to \cref{X and Y inside semidirect}, we are now able to settle
when a semidirect product of two L-algebras is self-similar and to explicitly compute the monoid operation in this case.
Moreover, with an inductive argument, we can extend any action of L-algebras to their self-similar closures.

\begin{pro}
\label{semidirect_and_selfsimilarity}
    Let $X$ and $Y$ be L-algebra such that $Y$ operates on $X$ via $\rho$.
    Then $X\rtimes_\rho Y$ is self-similar 
    if and only if $X$ and $Y$ are self-similar.
    Moreover, in this case, the monoid operation on $X\rtimes_\rho Y$ is given by
    \[
    (z,t)(x,u)=(z\rho_t(x),tu).
    \]
\end{pro}
\begin{proof}
Let's first assume that $X\rtimes_\rho Y$ is self-similar.
Then, in particular, for every $u\in Y$, the map
\[
\sigma_{(1,u)}:\dset(1,u)\to X\rtimes_{\rho} Y
\]
is bijective.
Hence, by \cref{X and Y inside semidirect}(\ref{part2}), the map
\[
\sigma_u:\dset u\to Y
\]
is also bijective.
Therefore, $Y$ is self-similar.
Let us now consider $x\in X$.
Since $X\rtimes_\rho Y$ is self-similar, 
the map
\[
\sigma_{(x,1)}:\dset(x,1)\to X\rtimes_{\rho} Y
\]
is bijective.
Thus, by \cref{X and Y inside semidirect}(\ref{part3}),
there exists a unique $(y,v)\in X\rtimes_\rho Y$ such that $y\in \dset \rho_{v}(x)$ and 
$(\rho_{v}(x)\cdot y,v)=(z,1)$.
But the $v$ is necessarily equal to 1.
Thus we proved that for every $z\in X$ there exists a unique $y\in X$ such that $y\in \dset x$ and 
$x\cdot y=\rho_{1}(x)\cdot y=z$.
Hence, the map
\[
\sigma_x:\dset x\to X
\]
is also bijective.

Suppose now that $X$ and $Y$ are self-similar and
let $(x,u),(z,t)\in X\rtimes_\rho Y$.
Since $Y$ is self-similar, there exists a unique $v\in \dset u$ such that $u\cdot v=\sigma_u(v)=t$.
Now, by self-similarity of $X$,
there exists a unique $y\in\dset\rho_{u\cdot v}(x)=\dset\rho_{t}(x)$ such that 
$\sigma_{\rho_{u\cdot v}(x)}(y)=z$.
Therefore, by \cref{X and Y inside semidirect}(\ref{part1}), we proved that there exists a 
unique $(y,v)\in\dset (x,u)$ such that 
\[
\sigma_{(x,u)}(y,v)
=\big(\rho_{u\cdot v}(x)\cdot \rho_{v\cdot u}(y), u\cdot v\big)
=\big(\sigma_{\rho_t(x)}(y), t\big)
=(z,t).
\]
Therefore, we proved that $X\rtimes_\rho Y$ 
is self-similar.
\end{proof}

\begin{pro}
\label{action of S(X)}
    Let $X$ and $Y$ be L-algebras such that $Y$ operates on $X$ via $\rho$.
    Then $S(Y)$ operates on $S(X)$ via a map 
    $\Tilde{\rho}$ such that 
    $\Tilde{\rho}_u(x)=\rho_u(x)$ for all $x\in X$ and $u\in Y$.
\end{pro}
\begin{proof}
    By \cref{Prop: universal property for S}, the functor $S$ is the left adjoint of the inclusion of the category of self-similar L-algebras in the category of L-algebras.
    So for every $u\in  Y$
    we have a map $S(\rho_u)\in \End(S(X))$
    that extends $\rho_u$.
    So we have a map $S(\rho):Y\to \End(S(X))$. 
    Consider now $M(Y)$ the free monoid generated by $Y$ (with identity element 1).
    By \cref{Thm: Extension for free monoid}, the operation $\cdot$ of $Y$ extends uniquely to $M(Y)$ such that 
    \begin{enumerate}
        \item $ab\cdot c=a\cdot(b\cdot c),$
        \item $a\cdot bc=((c\cdot a)\cdot b)(a\cdot c)$
        \item $1\cdot a=a$,
    \end{enumerate}
    for all $a,b,c\in M(Y)$.
  Since $M(Y)$ is the free monoid generated by $Y$, we can extend the map $S(\rho)$
to a map $\rho' \colon M(Y) \to \End(S(X))$.
In this way, we have $\rho'_1 = \id$, and we will prove that the second property,
\[
\rho'_{a \cdot b} \circ \rho'_a = \rho'_{b \cdot a} \circ \rho'_b,
\]
holds for all $a, b \in M(Y)$.

First, we show that the identity
\begin{equation} \label{eq:induction-step}
\rho'_{a \cdot y} \circ \rho'_a = \rho'_{y \cdot a} \circ \rho_y
\end{equation}
holds for all $a \in M(Y)$ and $y \in Y$. We proceed by induction on the length of $a$.

For the base case $a = 1$, identity \eqref{eq:induction-step} becomes
\[
\rho'_y \circ \rho'_1 = \rho'_{1} \circ \rho_y,
\]
which holds trivially since $\rho'_1 = \id$.

Assume that \eqref{eq:induction-step} holds for a word $a \in M(Y)$ of length $n \geq 1$. Let $x \in Y$. By the extension property given in \cref{Thm: Extension for free monoid}, we compute:
\begin{align*}
\rho'_{x a \cdot y} \circ \rho'_{x a} 
&= \rho'_{x \cdot (a \cdot y)} \circ \rho_x \circ \rho'_a \\
&= \rho'_{(a \cdot y) \cdot x} \circ \rho'_{a \cdot y} \circ \rho'_a \\
&= \rho'_{(a \cdot y) \cdot x} \circ \rho'_{y \cdot a} \circ \rho_y \\
&= \rho'_{((a \cdot y) \cdot x) \cdot (y \cdot a)} \circ \rho_y \\
&= \rho'_{y \cdot (x a)} \circ \rho_y.
\end{align*}
Hence, \eqref{eq:induction-step} is verified for $x a$ as well, completing the induction.

Now we proceed to prove the second property
\[
\rho'_{a \cdot b} \circ \rho'_a = \rho'_{b \cdot a} \circ \rho'_b
\]
by induction on the length of $b$. When $b$ has length one, this is just \eqref{eq:induction-step}. Suppose the property holds for a given $b$ of length $n \geq 1$. Let $x \in Y$. Then:
\begin{align*}
\rho'_{a \cdot (x b)} \circ \rho'_a 
&= \rho'_{((b \cdot a) \cdot x) \cdot (a \cdot b)} \circ \rho'_a \\
&= \rho'_{(b \cdot a) \cdot x} \circ \rho'_{a \cdot b} \circ \rho'_a \\
&= \rho'_{x \cdot (b \cdot a)} \circ \rho_x \circ \rho'_b \\
&= \rho'_{x b \cdot a} \circ \rho'_{x b},
\end{align*}
completing the induction step. Hence, the second property holds for all $a, b \in ~M(Y)$.

    We now verify that it can also be defined on $S(X)=M(Y)/{\approx}$.
    Let $a,b\in M(Y)$ such that $a\approx b$.
    Then, in particular $a\cdot b=b\cdot a=1$, hence
    \[
    \rho'_a=\rho'_1\rho'_a=\rho'_{a\cdot b}\rho'_a=\rho'_{b\cdot a}\rho'_b=\rho'_1\rho'_b=\rho'_b.
    \]
    Therefore we have a well-defined map $\Tilde{\rho}:S(X)\to \End(S(X))$ that extends
    $\rho$ and such that $\Tilde{\rho}_{a\cdot b}\Tilde{\rho}_a=\Tilde{\rho}_{b\cdot a}\Tilde{\rho}_b$ for all $a,b\in S(X)$.
\end{proof}

We are now ready to prove the main result of this section.

\begin{thm}\label{Thm: Gothere}
    Let $X$ and $Y$ be L-algebras such that $Y$ operates on $X$ via $\rho$.
    Then 
    \[
    S(X\rtimes_\rho Y)=S(X)\rtimes_{\Tilde{\rho}} S(Y).
    \]
\end{thm}
\begin{proof}
   By \cref{semidirect_and_selfsimilarity}, we know that 
$S(X)\rtimes_{\widetilde{\rho}} S(Y)$ is a self-similar L-algebra.  
Since the natural maps $X \to S(X)$ and $Y \to S(Y)$ are injective, the induced map  
\[
X \rtimes_{\rho} Y \longrightarrow S(X)\rtimes_{\widetilde{\rho}} S(Y)
\]
is also injective. Moreover, since $X$ generates $S(X)$ as a monoid and $Y$ generates $S(Y)$ as a monoid, for all
\((a,b) \in S(X)\rtimes_{\widetilde{\rho}} S(Y)\) there exist elements 
\(x_1,\dots,x_n \in X\) and \(u_1,\dots,u_m \in Y\) such that
\[
a = x_1 x_2 \cdots x_n 
\quad\text{and}\quad 
b = u_1 u_2 \cdots u_m.
\]

    Thus, by \cref{semidirect_and_selfsimilarity},
    \begin{align*}
        (a,b)&
    =(a,1)(1,b)
    =(x_1x_2\cdots x_n,1)(1,u_1u_2\cdots u_m)\\
    &=(x_1,1)(x_2,1)\cdots(x_n,1)(1,u_1)(1,u_2)\cdots(1,u_m).
    \end{align*}
    Since $(x_i,1),(1,u_j)\in X\rtimes_\rho Y$
    for every $1\leq i\leq n$ and $1\leq j\leq m$,
    we proved that $S(X)\rtimes_{\Tilde{\rho}} S(Y)$ is generated by $X\rtimes_\rho Y$
    as a monoid.
\end{proof}

\begin{cor}
      Let $X$ and $Y$ be KL-algebras such that $Y$ operates on $X$ via $\rho$ as KL-algebras. Then $S(X)\rtimes_{\Tilde{\rho}} S(Y)$ is also a KL-algebra.
\end{cor}
\begin{proof}
    By \cref{Prop: Self-silimar for KL-algebras}, $S(X\rtimes_\rho Y)=S(X)\rtimes_{\Tilde{\rho}} S(Y)$ is a KL-algebra.
\end{proof}

\section{Ideals of Semidirect Products of L-Algebras}\label{Section: Ideals of Semidirect Products of L-Algebras}
This section is dedicated to the study of ideals of a semidirect product.
Hence, we will consider two L-algebras $X$ and $Y$
such that $Y$ acts on $X$ via 
$\rho : Y \rightarrow \End(X)$.

\begin{defn}
    Let $K$ be an ideal of $X\rtimes_\rho Y$.
    We define 
    \begin{align*}
    K_Y&=\{y\in Y\mid (1,y)\in K\};\\ 
    K_X&=\{x\in X\mid (x,1)\in K\}.
    \end{align*}

    Similarly, if $X$ and $Y$ are CKL-algebras 
    such that $Y$ operates on $X$ via $\rho$ 
    and $L$ is an ideal of $X\sym_\rho Y$,
    we define
    \begin{align*}
    L_Y&=\{y\in Y\mid (1,y)\in L\};\\ 
    L_X&=\{x\in X\mid (x,1)\in L\}.
    \end{align*}
\end{defn}

\begin{lem}\label{Lemma:The subideals in semidirect product}   
    Let $K\subset X\rtimes_\rho Y$ be an ideal.
    Then $K_Y$ is an ideal of $Y$ and $K_X$
     is an ideal of $X$.
    
   Moreover, if $X$ and $Y$ are CKL-algebras such that $Y$ operates on $X$ via $\rho$ and $L\subset X\sym_\rho Y$ is an ideal, then $L_Y$ is an ideal of $Y$ and 
    $L_X$ is an ideal of $X$.
\end{lem}
\begin{proof}
Consider the following maps:
\[
f: X\to (X\rtimes_\rho Y)/K; ~ x\mapsto [(x,1)]_K \qquad g: Y\to (X\rtimes_\rho Y)/K; ~ y\mapsto [(1,y)]_K.
\]
It is easy to show that they are both L-algebra morphisms.
Moreover, $\ker f=K_X$ and $\ker g=K_Y$.
Therefore, $K_X$ is an ideal of $X$ and $K_Y$ is an ideal of $Y$.

Moreover, since $(x,1)$ and $(1,u)$ are elements of $X\sym_\rho Y$ for every $x\in X$ and $u\in Y$, we can apply the same strategy to prove that $L_X$ is an ideal of $X$ and $L_Y$ is an ideal of $Y$.
\end{proof}

\begin{lem}\label{Lemma:The ideal split of semidirect product}
Let $K\subset X\rtimes_\rho Y$ be an ideal.
Then $(x,u)\in K$ if and only if
$x\in K_X$ and $u\in K_Y$.

Moreover, if $X$ and $Y$ are CKL-algebras such that $Y$ operates on $X$ via $\rho$ and $L\subset X\sym_\rho Y$ is an ideal, then $(x,u)\in L$ 
if and only if $x\in L_X$ and $u\in L_Y$.
\end{lem}
\begin{proof}
    Suppose that $(x,u)\in K$.
    Then, 
    by \cref{X and Y inside semidirect}(\ref{part4}),
    $(1,u)\cdot(x,u)=(x,1)$, so $(x,1)\in K$.
    Moreover, $(x,u)\cdot(1,u)=(1,1)\in K$,
    hence $(1,u)$ is also in $K$.
    Suppose that $(x,1)$ and $(1,u)$ are in $K$.
    Then, 
    by \cref{X and Y inside semidirect}(\ref{part4}), 
     $(1,u)\cdot(x,u)=(x,1)\in K$, so $(x,u)\in K$.

    The same proof also works for the CKL-algebras case.
\end{proof}

\begin{cor}
\label{ideals of semidirect products are semidirect products}
     Let $K\subseteq X\rtimes_\rho Y$ be an ideal. We have that $$K=K_X\rtimes_{\rho|_{K_Y}} K_Y.$$
     
     If $X$ and $Y$ are CKL-algebra such that $Y$ operates on $X$ via $\rho$ and $L\subset X\sym_\rho Y$ is an ideal, then$$L=L_X\sym_{\rho|_{L_Y}} L_Y.$$
\end{cor}

\begin{proof}
We need to show that $\rho_u\in \End(K_X)$ for every $u\in K_Y$.
So let $u\in K_Y$ and $x\in K_X$, then
$\big(\rho_u(x),1\big)=(1,u)\cdot (x,1)$
and, since $(1,u),(x,1)\in K$ we obtain that $\big(\rho_u(x),1\big)\in K$, i.e. $\rho_u(x)\in K_X$.
Therefore $\rho_u\in \End(K_X)$, so 
$K_Y$ operates on $K_X$ via the restriction of $\rho$.

Moreover, by \cref{Lemma:The ideal split of semidirect product}, $K$ and $K_X\rtimes_{\rho|_{K_Y}} K_Y$ coincide as subsets of $X\rtimes_\rho Y$.

The same proof also works for the CKL-algebras case.
\end{proof}

Note that the converse of \cref{{ideals of semidirect products are semidirect products}} is not true, as shown by the following.

\begin{exa}
\label{examlple semidir}
    Let $X=\{1,x,y\}$ as in \cref{Exa: the Self-similar of CKL-algebra is not CKL-algebra}, i.e. $x\cdot y=y\cdot x=x $ 
    and let $Y=\{1,u\}$ be the L-algebra (unique of size 2 up to isomorphism) 
    with logical unit 1.
    Then
    $\End(X)=\{id, \sigma\}$, where $\sigma$ is the map that sends every element to 1.
    Moreover $X$ operates on $Y$ via $\rho:X\to \End(Y)$, where $\rho_1=id$ and $\rho_2=\sigma$.

    $X$ has two ideals:
    \[
    I_1=\{1\} \text{ and } I_2=X.
    \]
    $Y$ has only two ideals:
    \[
    J_1=\{1\} \text{ and } J_2=Y.
    \]
    But he semidirect product $X\rtimes_\rho Y$ has 3 ideals:
    \begin{align*}
         K_1&=\{(1,1)\}=I_1\rtimes_\rho J_1;\\
         K_2&=\{(x,1),(y,1),(1,1)\}=I_2\rtimes_\rho J_1;\\
         K_3&=X\rtimes_\rho Y=I_2\rtimes_\rho J_2.
    \end{align*}
    So, $I_1\rtimes_\rho J_2$ is not an ideal.
    For example
    \[
    (1,u)\cdot(x,1)=(\rho_{u\cdot 1}(1)\cdot\rho_{1\cdot u}(x),u\cdot 1)
    =(1\cdot \rho_u(x), 1)=(1,1)\in I_1\rtimes_\rho J_2
    \]
    but $(x,1)\notin \{1\}\rtimes_\rho J_2=I_1\rtimes_\rho J_2$.
\end{exa}

Now we know that ideals of the semidirect products
are semidirect products of ideals in the respective components.
However, as shown in \cref{examlple semidir}, there exist ideals in the individual components that cannot be used to form an ideal of the semidirect product.

The following proposition provides an equivalent characterization for when such pairs of ideals give rise to an ideal of the semidirect product.

\begin{pro}\label{Lemma:The equivalence of the ideal in semidirect product}  Let $X$ and $Y$ be L-algebras such that $Y$ operates on $X$ via $\rho$. Let $I$ be an ideal of $X$, $U$ be an ideal of $Y$. Then $I\rtimes_{\rho|_{U}} U$ is an ideal of $X\rtimes_{\rho} Y$ if and only if,   for each $u\in U$, $x\in I$, $v\in Y$, and $y\in X$ \begin{enumerate}[font=\upshape]    
\item[(I'1)] $\rho_v(I)\subseteq I$;    
\item[(I'2)] $(x\cdot\rho_u(y))\cdot y,~y\cdot 
 (x\cdot\rho_u( y))\in I$.
\end{enumerate}\end{pro}
\begin{proof}
    We fix that $u\in U$, $x\in I$, $v\in Y$, and $y\in X$. And, we denote a condition (I'3) for $\rho$ as: $\rho_u^{-1}(I)\subseteq I$ for $u\in U$.

    If $I\rtimes_{\rho|_{U}} U$ satisfies (I1),   $(1,v)\cdot (x,1)\in I\rtimes_{\rho|_{U}} U$. Then  we have $\rho_v(x)\in I$, that is (I'1). If $\rho$ satisfies (I'1), we have $(y,v)\cdot (x,u)=(\rho_{v\cdot u}(y)\cdot \rho_{u\cdot v}(x),v\cdot u)\in I\rtimes_{\rho|_{U}} U$. Thus, (I1) for $I\rtimes_{\rho|_{U}} U$ is equivalent to (I'1) for $\rho$.

    Assume that condition (I'1) holds for $\rho$. Let $$(x,u)\cdot (y,v)= (\rho_{ u\cdot v}(x)\cdot \rho_{ v\cdot u}(y), u\cdot v)\in I\rtimes_{\rho|_{U}}U.$$ Since $I$ is an ideal, we have $\rho_{ v\cdot u}(y)\in I$. Take $v$ to be $1$,   $\rho_{   u}(y)\in I$. Thus, (I2) for $I\rtimes_{\rho|_{U}} U$ is equivalent to the condition  (I'3). 

For $u\in U$, $x\in I$, $v\in Y$, and $y\in X$, we have
    \begin{align*}
        \left(\left(x,u\right)\cdot \left(y,v\right)\right)\cdot(y,v)&=(\rho_{ u\cdot v}(x)\cdot \rho_{ v\cdot u}(y), u\cdot v)\cdot (y,v)\\
        &=\Big(\underbrace{\left(\rho_{(u\cdot v)\cdot v}\rho_{u\cdot v} (x)\cdot \rho_{(u\cdot v)\cdot v}\rho_{v\cdot u}(y)\right)\cdot \rho_{v\cdot (u\cdot v)}(y)}_{A},(u\cdot v)\cdot v\Big).
    \end{align*}
       We denote the first component of $\left(\left(x,u\right)\cdot \left(y,v\right)\right)\cdot(y,v)$  as $A$. Notice that, since $(u\cdot v)\cdot v\in U$, $\left(\left(x,u\right)\cdot \left(y,v\right)\right)\cdot(y,v)\in I\rtimes_{\rho|_{U}} U$ if and only if $A\in I$. Now, we suppose $A\in I$ and choose $x=\rho_u(x')$, $v=1$, $y=\rho_u(y')$ with $x'\in I$, $y'\in Y$. Then $A= \rho_u((x'\cdot \rho_u(y'))\cdot y') \in I$. By (I'3), we have $(x'\cdot \rho_u(y'))\cdot y'\in I$.  Meanwhile, we have 
       \begin{align*}
        (y,v)\cdot\left(\left(x,u\right)\cdot \left(y,v\right)\right)&=(y,v)\cdot(\rho_{ u\cdot v}(x)\cdot \rho_{ v\cdot u}(y), u\cdot v) \\
        &=\Big(\underbrace{\rho_{ v\cdot(u\cdot v}(y)\cdot \left(\rho_{(u\cdot v)\cdot v}\rho_{u\cdot v} (x)\cdot \rho_{(u\cdot v)\cdot v}\rho_{v\cdot u}(y)\right)}_{B},v\cdot(u\cdot v) \Big).
    \end{align*}
  Similarly, $y'\cdot (x'\cdot \rho_u(y'))\in I$. Thus, (I'2) holds for $\rho$, if $I\rtimes_{\rho|_{U}} U$ is an ideal of $X\rtimes_\rho Y$. Now, if $I\rtimes_{\rho|_{U}} U$ is an ideal of $X\rtimes_\rho Y$, the three conditions hold. 

  Let now assume that (I'1),  and (I'2) hold for $\rho$ and for every $x\in I$ and $u\in U$. Firstly, we show that $\rho_{v\cdot u}(y)\cdot \rho_{v\cdot( u\cdot v)}(y)$ and $\rho_{v\cdot( u\cdot v)}(y)\cdot \rho_{v\cdot u}(y)\in I$. Denote the following expressions, respectively, by C, D, E, and F:
\begin{align*}
C &= \rho_{v\cdot (u\cdot v)}(y)\cdot\left(x\cdot \rho_{((u\cdot v)\cdot v)\cdot (v\cdot u)}\rho_{v\cdot (u\cdot v)}(y)\right), \\
D &= \left( x \cdot \rho_{((u \cdot v) \cdot v) \cdot (v \cdot u)} \rho_{v \cdot (u \cdot v)}(y) \right) \cdot \rho_{v \cdot (u \cdot v)}(y), \\
E &= \rho_{v \cdot u}(y) \cdot \left( x \cdot \rho_{((u \cdot v) \cdot v) \cdot (v \cdot u)} \rho_{v \cdot u}(y) \right), \\
F &= \left( x \cdot \rho_{((u \cdot v) \cdot v) \cdot (v \cdot u)} \rho_{v \cdot u}(y) \right) \cdot \rho_{v \cdot u}(y).
\end{align*} By (I'1) and (I'2), we have that $C,D, E, F\in I$. Thus, by 
   $$C\cdot (\rho_{v\cdot u}(y)\cdot \rho_{v\cdot( u\cdot v)}(y))=D\cdot F$$ and $$E(x)\cdot (\rho_{v\cdot u}(y)\cdot \rho_{v\cdot( u\cdot v)}(y))=F\cdot D,$$ we have
    $\rho_{v\cdot u}\cdot \rho_{v\cdot( u\cdot v)}(y)$ and $\rho_{v\cdot( u\cdot v)}(y)\cdot \rho_{v\cdot u}(y)\in I$.

  Denote $\left(\rho_{(u\cdot v)\cdot v}\rho_{u\cdot v} (x)\cdot \rho_{(u\cdot v)\cdot v}\rho_{v\cdot u}(y)\right)\cdot \rho_{v\cdot u}(y)$ by $G$. Notice that 
  $$G\cdot A=\left(\rho_{v\cdot u}\cdot (\rho_{(u\cdot v)\cdot v}\rho_{u\cdot v}(x)\cdot \rho_{(u\cdot v)\cdot v}\rho_{v\cdot u}(y))\right)\cdot \left(\rho_{v\cdot u}(y)\cdot \rho_{v\cdot( u\cdot v)}(y)\right), $$ 
 and  $\rho_{v\cdot u}(y)\cdot \rho_{v\cdot( u\cdot v)}(y),G\in I$. Thus, $A\in I$.

Denote $\rho_{v\cdot u}(y)\cdot \left(\rho_{(u\cdot v)\cdot v}\rho_{u\cdot v} (x)\cdot \rho_{(u\cdot v)\cdot v}\rho_{v\cdot u}(y)\right)$ by $H$. Similarly, 
\begin{align*}
    H\cdot B=&\left( \left(\rho_{(u\cdot v)\cdot v}\rho_{u\cdot v} (x)\cdot \rho_{(u\cdot v)\cdot v}\rho_{v\cdot u}(y)\right)\cdot\rho_{v\cdot u}(y)\right)\\
    &\cdot \left(  \left( \rho_{(u\cdot v)\cdot v}\rho_{v\cdot u}(y)\cdot 
  \rho_{(u\cdot v)\cdot v}\rho_{u\cdot v} (x)\right)\cdot\left( \rho_{(u\cdot v)\cdot v}(\rho_{v\cdot( u\cdot v)}(y)\cdot \rho_{v\cdot u}(y))
      \right)\right). 
\end{align*}
Since $ \rho_{v\cdot( u\cdot v)}(y)\cdot \rho_{v\cdot u}(y)$ and $G\in I$, we can show that $B\in I$.

Take $x$ to be $1$. From (I'2), we have $ \rho_u(y)\cdot y\in I$ for all $u\in U$. Thus, if $\rho_u(y)\in I$, then we obtain $y\in I$. Therefore, (I'2) implies (I'3) for $\rho$. and also implies (I2) for $I\rtimes_{\rho|_{U}} U$.
\end{proof}

As a consequence, we obtain an explicit upper bound for the number of ideals of the semidirect product of two L-algebras.

\begin{cor}
   Let $X$ and $Y$ be L-algebras such that $Y$ operates on $X$ via $\rho$.  Let $U$ be an ideal of $Y$.
     Then $\{1\}\rtimes U$ is an ideal of $X\rtimes_\rho Y$ if and only if $\rho_u=\id_X$ for all $u\in U$. In particular, 
     \[
     |\mathscr{I}(X\rtimes_\rho Y)|\leq|\mathscr{I}(X)||\mathscr{I}(Y)|,
     \]
     and the equality holds if and only if $X\rtimes_\rho Y=X\times Y$.
\end{cor}
\begin{proof} 
    If $\rho_u=\id_X$ for all $u \in U$, then it is easy to check that the conditions (I'1) and (I'2)
    of \cref{Lemma:The equivalence of the ideal in semidirect product} are satisfied.

    Vice versa, if $\{1\}\rtimes U$ is an ideal of $I\rtimes_{\rho|_{U}} U$, then,
    by condition (I'2) of 
    \cref{Lemma:The equivalence of the ideal in semidirect product},
    we get that $(1\cdot\rho_u(y))\cdot y,  y\cdot (1\cdot\rho_u( y))\in \{1\}$
    for all $u\in U$ and $y\in X$.
    Therefore, for all $u\in U$ and $y\in X$
    \[
    \rho_u(y)\cdot y=1=y\cdot \rho_u(y).
    \]
 which implies that $\rho_u(y) = y$ for all $u\in U$ and $y\in X$.

   The second property derives from the fact that if $\rho_y\neq \id$,
   then $\{1\}\rtimes_\rho Y$ is not an ideal of $X\rtimes_\rho Y$.
\end{proof}

By \cref{Lemma:The subideals in semidirect product}, \cref{ideals of semidirect products are semidirect products} \cref{Lemma:The equivalence of the ideal in semidirect product} and \cref{def: symmetric semi-direct product}, we can immediately obtain the following result.

\begin{thm}\label{Theorem:The equivalence of the ideal in semidirect product}
  Let $X$ and $Y$ be L-algebras such that $Y$ operates on $X$ via $\rho$.  
Then $K$ is an ideal of $X \rtimes_{\rho} Y$ if and only if $\rho$ induces an operation 
\[
\tilde{\rho} : Y / K_Y \longrightarrow \End(X / K_X)
\]
such that
\[
(X \rtimes_{\rho} Y)\big/ (K_X \rtimes_{\rho|_{K_Y}} K_Y )
\;\cong\;
X / K_X \rtimes_{\tilde{\rho}} Y / K_Y .
\]

\end{thm}

\begin{proof} 
By \cref{Lemma:The subideals in semidirect product} and \cref{ideals of semidirect products are semidirect products}, we have
\[
K = K_X \rtimes_{\rho|_{K_Y}} K_Y,
\]
where $K_X$ is an ideal of $X$ and $K_Y$ is an ideal of $Y$.

Next, set $I = K_X$ and $U = K_Y$.  

By (I'1) of \cref{Lemma:The equivalence of the ideal in semidirect product}, the map 
\[
\rho^I: Y \to \End(X/I)
\]
is well-defined. Since $I \rtimes \{1\}$ is an ideal of $X \rtimes_\rho Y$, the quotient
\[
(X \rtimes_\rho Y) / (I \rtimes \{1\}) \cong X/I \rtimes_{\rho^I} Y
\]
is an L-algebra. By (I'2) of \cref{Lemma:The equivalence of the ideal in semidirect product}, for all $u \in U$ and $y \in X/I$, we have
\[
\rho^I_u(y) \cdot y = [1]_I = y \cdot \rho^I_u(y),
\]
which implies $\rho^I_u = \id_{X/I}$ for all $u \in U$. Moreover, for all $v,w \in U$,
\[
\rho^I_v = \rho^I_{v \cdot w} \circ \rho^I_v = \rho^I_{w \cdot v} \circ \rho^I_w = \rho^I_w,
\]
so the induced map
\[
\tilde{\rho}: Y/U \to \End(X/I)
\]
is well-defined.

Conversely, if $\rho^I$ is well-defined, then (I'1) holds for $\rho$. Since $\rho^I_u = \id_{X/I}$ for all $u \in U$, it follows that 
\[
(x \cdot \rho_u(y)) \cdot y, \quad y \cdot (x \cdot \rho_u(y)) \in I
\]
for all $x,y \in X$ and $u \in U$.
\end{proof}

We can also directly imply the following corollary.

\begin{cor}\label{Cor:The equivalence of the ideal in semidirect product}
 Let $X$ and $Y$ be L-algebras such that $Y$ operates on $X$ via $\rho$.    Let $I$ be an ideal of $X$, $U$ be an ideal of $Y$.  The following statements are equivalent:
\begin{enumerate}
    \item $I\rtimes_{\rho|_{U}} U$ is an ideal of $X\rtimes_\rho Y$;
    \item $\rho$ can induce an operation $\tilde{\rho}:Y\slash U\to \End(X\slash I)$ such that  $$(X\rtimes_\rho Y)\slash (I\rtimes_{\rho|_{U}} U ) \cong X\slash I\rtimes_{\tilde{\rho}} Y\slash U;$$
    \item $\rho$ induces an operation $\rho^I:Y\to \End(X\slash I)$ such that 
    \[X\slash I\rtimes_{\rho^I|_U} U= X\slash I\times U;\]
    \item $\rho$ can induce an operation $\rho^I:Y\to \End(X\slash I)$ such that $  \{[1]_I\}\times U$ is an ideal of $X\slash I\rtimes_{\rho^I} Y$.
\end{enumerate}  
\end{cor}

We now focus on the L-algebra structure of the set of ideals.
The next example shows how, in general, it does not commute with the semidirect product.

\begin{exa}
Using the same L-algebras as in \cref{examlple semidir}, we obtain
\[
\langle 1\rangle = \{1\} = I_1, 
\qquad 
\langle u\rangle = Y = J_2,
\qquad 
\langle (1,u)\rangle = K_3 = I_2 \rtimes J_2 = I_2 \rtimes \langle u\rangle.
\]

Moreover, since $K_2 \cdot K_1 = K_1$, we obtain that
\[
(I_2 \rtimes J_1)\cdot (I_1 \rtimes J_1)
= K_2 \cdot K_1
= K_1
= I_1 \rtimes J_1
\neq
I_1 \rtimes J_2
= (I_2 \cdot I_1) \rtimes (J_1 \cdot J_1).
\]
\end{exa}

However, we can still deduce some structural properties about the product of two ideals of the semidirect product
and completely determine it in some specific cases.

\begin{lem}\label{product between Xx1 and any ideal}
Let $X$ and $Y$ be L-algebras such that $Y$ operates on $X$ via $\rho$. Then the following statements hold.
\begin{enumerate}

    \item
    Let $(x,u)\in X\rtimes_\rho Y$ and let $K=\langle (x,u)\rangle$ be the ideal generated by $(x,u)$. Then
    \[
        K = K_X \rtimes \langle u\rangle,
    \]
    and $\langle x\rangle \subseteq K_X$.\label{product between Xx1 and any ideal Part 1}

    \item
    Let $I\rtimes U$ and $J\rtimes V$ be ideals of $X\rtimes_\rho Y$, and let
    \[
        L = (I\rtimes U)\cdot (J\rtimes V).
    \]
    Then
    \[
        L_X \subseteq I\cdot J
        \qquad\text{and}\qquad
        L_Y \subseteq U\cdot V.
    \]\label{product between Xx1 and any ideal Part 2}

    \item
    Let $I\rtimes U$ be an ideal of $X\rtimes_\rho Y$. Then
    \[
        (X\rtimes \{1\})\cdot (I\rtimes U) = I\rtimes V,
    \]
    where $V$ is the maximal ideal of $Y$ such that $I\rtimes V$ is an ideal of $X\rtimes_\rho Y$, and where   \[
        V = \ker\bigl(\rho^I : Y \to \End(X/I)\bigr).
    \]\label{product between Xx1 and any ideal Part 3}

\end{enumerate}
\end{lem}

\begin{proof}~
\begin{enumerate}[font=\itshape]
    \item  By \cref{Theorem:The equivalence of the ideal in semidirect product},  $X\rtimes  \langle u\rangle$  is an ideal of $ X\rtimes_\rho Y$. So, we have $K\subseteq X\rtimes  \langle u\rangle $. Thus, $K_Y\subseteq \langle u\rangle$. In addition, $K_Y$ is an ideal of $Y$ that contains $u$. We have $K_Y= \langle u\rangle$.
    \item  By definition, $(x,u)\in L$ if and only if $\langle (x,u)\rangle\cap (I\rtimes U)\subseteq J\rtimes V$.
   By Part {\it(1)}, this means that 
   $\langle (x,u)\rangle_X\cap I \subseteq J$ and
   $\langle u\rangle \cap U\subseteq V$,
   i.e. 
   \[
   L=\{(x,u)\in (I\cdot J)\times (U\cdot V)\mid \langle (x,u)\rangle_X\cap I \subseteq J\}.
   \]
   \item Let $L$ be $(X\rtimes \{1\})\cdot ( I\rtimes U)$.
By Part {\it(2)}, we have $L_X\subseteq X\cdot I= I$. Moreover, $I\rtimes U\subseteq L$.
Hence $L_X=I$ and $L$ is the greatest ideal such that $L\cap (X\rtimes \{1\})\subseteq ( I\rtimes U)$, i.e. such that $L_X\subseteq I$.
Therefore, we prove the thesis.\qedhere
\end{enumerate}
    
\end{proof}

\begin{defn}Let $X$ and $Y$ be L-algebras such that $Y$ operates on $X$ via $\rho$.

\begin{itemize} 
    \item An ideal $I$ of $X$ is called \emph{$\rho$-ideal} if $I$ satisfies (I'1) i.e. $\rho_v(I)\subseteq I$ for every $v\in Y$. 
    \item   A proper $\rho$-ideal $I$ of $X$ is called \emph{$\rho$-prime} if 
    for every $\rho$-ideals $I_1$ and $I_2$ of $X$ and such that $I_1\cap I_2\subseteq I$,
        then either $I_1\subseteq I$ or $I_2\subseteq I$.
\end{itemize}
We denote by $\rho\mathscr{I}(X)$ the poset of $\rho$-ideals of $X$, and by $\rho$-spectrum  $\rho\Spec(X)$ the space of $\rho$-prime ideals.
   
\end{defn}

Note that the name is not an accident.
Indeed, a proper ideal $I$ is prime if and only if for every ideals $I_1$ and $I_2$ of $X$ such that $I_1\cap I_2\subseteq I$, either $I_1\subseteq I$ or $I_2\subseteq I$.
So if $I$ is prime and satisfies (I'1), it is also $\rho$-prime.

\begin{lem}\label{Maxmial ideal in Y for semiproduct}
    Let $J$ and $I$ be two $\rho$-ideals of $X$.
    Let $\rho^I:Y\to \End(X/I)$ and $\rho^J:Y\to \End(X/J)$ be the maps induced by $\rho$.
    Then the following statements hold:
\begin{enumerate}
    \item $I \rtimes \ker(\rho^I) $ (and  $J \rtimes \ker(\rho^I) $) is an ideal of the semi-direct product $ X \rtimes_\rho Y $.
    \item If $ J \subseteq I $, then $\ker(\rho^J) \subseteq \ker(\rho^I)   $.
    \item If $ U \subset \ker(\rho^I) $ is an ideal of $ Y $, then $I \rtimes U$ is an ideal of $ X \rtimes_\rho Y $.
\end{enumerate}

\end{lem}
\begin{proof}~
    \begin{enumerate}[font=\itshape]
        \item By \cref{Lemma:The equivalence of the ideal in semidirect product}, $I\rtimes \ker(\rho^I)$ is an ideal of $X\rtimes_\rho Y$ if and only if $I$ satisfies (I'2).
        Let $u\in \ker(\rho^I)$, $x\in I$ and $y\in X$,
        then in $X/I$ we have that
        \begin{align*}
            \big[(x\cdot \rho_u(y))\cdot y\big]_I&
            =([x]_I\cdot [\rho_u(y)]_I)\cdot [y]_I\\
            &=(1\cdot \rho^I_u([y]_I))\cdot [y]_I
            =(1\cdot [y]_I)\cdot [y]_I=1.
        \end{align*}

        Hence $(x\cdot \rho_u(y))\cdot y\in I$.
        Similarly,
        \begin{align*}
            \big[y \cdot (x\cdot \rho_u(y))\big]_I&
            =[y]_I \cdot ([x]_I\cdot [\rho_u(y)]_I)\\
            &=[y]_I \cdot (1\cdot \rho^I_u([y]_I))
            =[y]_I \cdot (1\cdot [y]_I)=1,
        \end{align*}
        i.e. $ y \cdot (x\cdot \rho_u(y))\in I$.
        Thus, we proved that (I'2) is satisfied too,
        i.e. $I\rtimes \ker(\rho^I)$ is an ideal of $X\rtimes_\rho Y$.
        
        \item Take $u\in \ker(\rho^I)$, then for every $x\in X$
        \[
        \rho^J_u\big([x]_J\big)=[\rho_u(x)]_J=[x]_J.
        \]
        Hence, $\rho_u(x)\in [x]_J$
        for every $x\in X$.
        Since $J\subseteq I$, then $[x]_J\subseteq [x]_I$ for every $x\in X$.
        Therefore $\rho_u(x)\in [x]_I$ for every $x\in X$.
        So $[\rho_u(x)]_I=[x]_I$ for every $x\in X$, which means that $\rho^I_u=id_{X/J}$, i.e. $u\in \ker(\rho^I)$.
        \item  This is clear that $I$ and $U$ satisfy conditions (I'1) and (I'2).\qedhere
    \end{enumerate}
\end{proof}

\begin{cor}
     Let $X$ and $Y$ be L-algebras such that $Y$ operates on $X$ via $\rho$. Then 
       \[
     |\mathscr{I}(X\rtimes_\rho Y)|= \sum_{I\in \rho\mathscr{I}(X)}|\{U\leq \ker(\rho^I)~|~U\in \mathscr{I}(Y)\}|.
     \]
     
\end{cor}

\begin{pro}
\label{rhoI is a sublattice of I}
 Let $X$ and $Y$ be L-algebras such that $Y$ operates on $X$ via $\rho$. Then \[\rho \mathscr{I}(X)= \{I\in \mathscr{I}(X)\mid I\rtimes_{\rho}\{1\}\in \mathscr{I}(X\rtimes_{\rho} Y)\}.\]
 Moreover, $\rho \mathscr{I}(X)$ is a complete sublattice of $\mathscr{I}(X)$, and is distributive. 
\end{pro}

\begin{proof}
Let $I$ be a $\rho$-ideal and $U=\ker(\rho^I)$.  By \cref{Maxmial ideal in Y for semiproduct}, therefore, $I\rtimes_{\rho}\{1\}$ is an ideal of $X\rtimes_{\rho} Y$. Vice versa, if $I\rtimes_{\rho}\{1\}$ is an ideal of $X\rtimes_{\rho} Y$, $I$ satisfies (I'1). Therefore, $\rho \mathscr{I}(X)= \{I\mid I\rtimes_{\rho}\{1\}\in \mathscr{I}(X\rtimes_{\rho} Y)\}$.

Next, we will show that $\rho \mathscr{I}(X)$ is a complete sublattice.

    Let $\{I_{\alpha}~|~\alpha\in \mathscr{Z}\}$ be a set of $\rho$-ideals and $v\in Y$. Then $\rho_v(\cap_{\alpha\in \mathscr{Z}} I_\alpha)\subseteq I_\alpha$, for each $\alpha\in \mathscr{Z}$. Thus, the $\rho$-ideals are closed with respect to intersections.
    
    Let $I_1$ and $I_2$ be $\rho$-ideals, $y\in I_1\vee I_2$ and $v\in Y$. By \cref{Thm: The elements in joint ideals}, there exists an element $x \in I_1$ with $x \equiv y \pmod{I_2}$. Since $\rho_v(I_1)\subseteq I_1$ and $\rho_v( I_2)\subseteq I_2$, then $\rho_v(x) \in I_1$ and $\rho_v(x) \equiv \rho_v(y)\pmod{I_2}$. Then $\rho_v(y)\in I_1\cap I_2$. Thus, the $\rho$-ideals are closed with respect to the joints. 
 Therefore,   $\rho\mathscr{I}(X)$ is a complete sublattice of $\mathscr{I}(X)$ .
\end{proof}

\begin{pro}~
\label{some prime ideals of semidirect}
\begin{enumerate}
    \item Let $U$ be a prime ideal of $Y$, then $X\rtimes U$ is a prime ideal of $X\rtimes_\rho Y$.
    \item Let $I$ be a $\rho$-prime ideal of $X$ and $U=\ker(\rho^I)$.
    Then $I\rtimes U$ is a prime ideal of $X\rtimes_\rho Y$.
\end{enumerate}    
\end{pro}
\begin{proof}~
\begin{enumerate}[font=\itshape]
    \item By \cref{Theorem:The equivalence of the ideal in semidirect product},
    $(X\rtimes_\rho Y)/(X\rtimes U)\cong Y/U$.
    Since $U$ is a prime ideal,
    $Y/U$ is subdirectly irreducible.
    Thus $X\rtimes U$ is a prime ideal of $X\rtimes_\rho Y$.
    \item 
    Let $I_1\rtimes U_1$ and $I_1\rtimes U_1$ be ideals of $X\rtimes_\rho Y$, such that 
    $$(I_1\rtimes U_1)\cap (I_2\rtimes U_2)\subseteq I\rtimes U.$$
    Thus, we have $I_1\cap I_2\subseteq I$. Since $I_1$ and $I_2$ satisfy (I'1), then either $I_1\subseteq I$ or $I_2\subseteq I$. By \cref{product between Xx1 and any ideal} (\ref{product between Xx1 and any ideal Part 3}), either $I_1\rtimes U_1\subseteq I\rtimes U$ or $I_2\rtimes U_2\subseteq I\rtimes U$. Therefore, $I\rtimes U$ is a prime ideal.\qedhere
\end{enumerate}\end{proof}

\begin{thm}\label{Thm: prime ideals of semidirect product of L-algebras}
    Let $P$ be an ideal of $X$ and $Q$ be an ideal of $Y$.
    $P\rtimes Q$ is a prime ideal of $X\rtimes_\rho Y$ if and only if one of the following holds:
    \begin{enumerate}
        \item $P=X$ and $Q$ is a prime ideal of $Y$;
        \item $P$ is a $\rho$-prime ideal of $X$ and $Q=\ker(\rho^P)$.
    \end{enumerate}
    
    Moreover, in this case,
\[
\Spec(X \rtimes_\rho Y) \cong \rho\Spec(X) \sqcup \Spec(Y).
\]

\end{thm}
\begin{proof}
   By \cref{some prime ideals of semidirect}, we know that in both cases we obtain a prime ideal of $X \rtimes_\rho Y$.

Now suppose that $P \rtimes Q$ is a prime ideal.  
Since $X \rtimes \{1\}$ is an ideal of $X \rtimes_\rho Y$, we must have either
\[
    X \rtimes \{1\} \subseteq P \rtimes Q
    \qquad\text{or}\qquad
    (X \rtimes \{1\}) \cdot (P \rtimes Q) \subseteq P \rtimes U.
\]

In the first case, $X \rtimes \{1\} \subseteq P \rtimes Q$, which implies $P = X$.  
Moreover, $Q$ is a prime ideal of $Y$ because $Y/P$ is subdirectly irreducible: indeed,
$(X \rtimes_\rho Y)/(P \rtimes Q)$ is subdirectly irreducible and, by 
\cref{Theorem:The equivalence of the ideal in semidirect product},
\[
    (X \rtimes_\rho Y)/(X \rtimes P) \cong Y/P.
\]

    In the second case,
    $P\rtimes Q\supseteq(X\rtimes \{1\})\cdot ( P\rtimes Q) $
    and by \cref{product between Xx1 and any ideal} (\ref{product between Xx1 and any ideal Part 3}),
    \[
    (X\rtimes \{1\})\cdot ( P\rtimes Q) = P\rtimes \ker(\rho^P),
    \]
    where $V=\ker(\rho^P)$ is an ideal of $Y$ that is maximal such that $P\rtimes V$ is an ideal of $X\rtimes_\rho Y$.
    But inclusion $P\rtimes Q\supseteq P\rtimes V$ forces $Q=\ker(\rho^P)$.

    Consider ideals $I_1$ and $I_2$ of $X$ that satisfy (I'1) and such that  
$P_1 \cap P_2 \subseteq P$. Then $P_1 \cap P_2$ also satisfies (I'1).  
By \cref{Maxmial ideal in Y for semiproduct}, we have
\[
    \ker(\rho^P)
    \subseteq \ker(\rho^{P_1 \cap P_2})
    = \ker(\rho^{P_1}) \cap \ker(\rho^{P_2}).
\]
Hence,
\[
    \bigl(P_1 \rtimes \ker(\rho^{P_1})\bigr)
    \cap 
    \bigl(P_2 \rtimes \ker(\rho^{P_2})\bigr)
    \subseteq
    \bigl(P \rtimes \ker(\rho^P)\bigr).
\]
Since $P \rtimes \ker(\rho^P)$ is prime, we must have either 
\[
    P_1 \rtimes \ker(\rho^{P_1}) \subseteq P \rtimes \ker(\rho^P)
    \quad\text{or}\quad
    P_2 \rtimes \ker(\rho^{P_2}) \subseteq P \rtimes \ker(\rho^P).
\]
Thus either $P_1 \subseteq P$ or $P_2 \subseteq P$.
This proves that $P$ is a $\rho$-prime ideal.
\end{proof}

\begin{rem}
    The bijection
\[
f : \Spec(X \rtimes_{\rho} Y)\longrightarrow \rho\Spec(X) \sqcup \Spec(Y),
\]
defined by
\[
f(X \rtimes Q)=Q 
\quad\text{and}\quad 
f(P \rtimes Q)=P \ \text{for } P \neq X,
\]
is an open map, where the subspace $\rho\Spec(X)$ is endowed with the subspace topology inherited from $\Spec(X)$.
\end{rem}
\begin{proof}
  
Let $I\rtimes U$ be an ideal of $X\rtimes_{\rho}Y$ with $I\neq X$, and let $P\rtimes Q$ be a prime ideal of $X\rtimes_{\rho}Y$. By \cref{Maxmial ideal in Y for semiproduct},
\[
I\rtimes U \subseteq P\rtimes Q 
\quad \Longleftrightarrow \quad 
I \subseteq P.
\]
Therefore,
\[
f(\mathscr{U}_{I\rtimes U})
= (\mathscr{U}_{I} \cap \rho\Spec(X)) \sqcup \mathscr{U}_{U},
\qquad
f(\mathscr{U}_{X\rtimes_\rho V})
= \rho\Spec(X) \sqcup \mathscr{U}_{V},
\quad (V\in \mathscr{I}(Y)),
\]
which shows that $f$ is an open map. 
\end{proof}

By \cref{Thm: prime ideals of semidirect product of L-algebras}, we can explicitly describe the prime spectrum of the semidirect product of KL-algebras as follows.
\begin{pro}\label{Thm: prime ideals of semidirect product of KL-algebras}
    Let $X$ and $Y$ be KL-algebras such that $Y$ operates on $X$ via $\rho$ as KL-algebras. Then  \[\rho \mathscr{I}(X)= \mathscr{I}(X).\]
      In this case, \[\rho\Spec(X) = \Spec(X) \quad \text{and} \quad \Spec(X\rtimes_\rho Y)\cong \Spec(X)\sqcup\Spec(Y).\]
\end{pro}

\begin{proof}By \cref{Def: operate on KL-alg}, we have
\[
x \cdot \rho_u(x) = 1,
\]
for all $ x \in X $ and $ u \in Y $. It follows that for any $u \in Y$,
\[
x \cdot \rho_u(x) = 1 \in \langle x\rangle.
\]
By condition~(I1), we conclude that 
\[
\rho_u(x) \in \langle x\rangle 
\quad\text{for all }u \in Y.
\]

Therefore, every ideal of $ X $ is automatically a $\rho$-ideal. By \cref{Thm: prime ideals of semidirect product of L-algebras}, we obtain the following characterizations:
\[
\rho\mathscr{I}(X) = \mathscr{I}(X), \quad \text{and} \quad \rho\Spec(X) = \Spec(X).\qedhere
\]
\end{proof}

\begin{pro}\label{Thm: ideals of symmetric semidirect product}
    Let $X$ and $Y$ be CKL-algebras such that $Y$ acts on $X$ via $\rho$ as CKL-algebras. 
    Let $L$ be an ideal of the symmetric semidirect product $X \sym_\rho Y$. Define 
    \[
        \widetilde{L} = L_{X} \rtimes_{\rho|_{L_Y}} L_Y \subseteq X \rtimes_\rho Y.
    \]
    Then the assignments 
    \[
        L \longmapsto \widetilde{L}
        \quad \text{and} \quad
        K \cap (X \sym_\rho Y)\longmapsfrom  K
    \]
    establish a bijective correspondence between the ideals of the symmetric semidirect product $X \sym_\rho Y$ and those of the semidirect product $X \rtimes_\rho Y$.
\end{pro}

\begin{proof}
    Let $L$ be an ideal of $X \sym_\rho Y$.  
    By \cref{Thm: prime ideals of semidirect product of KL-algebras}, every ideal of $X$ is a $\rho$-ideal, hence $L_X$ satisfies condition (I'1). 

    Let $x \in L_X$, $u \in L_Y$, and $y \in X$.  
    By \cref{Def: operate on CKL-alg}(\ref{Def: operate on CKL-alg part 2}), we have 
    \[
        y \cdot (x \cdot \rho_u(y)) = \rho_u\bigl(y \cdot (x \cdot y)\bigr) \in L_X.
    \]
    Moreover, since $(x,u) \in L$, $(y,1) \in X \sym_\rho Y$, and $L$ is an ideal of $X \sym_\rho Y$, it follows that 
    \[
        \bigl((x,u) \cdot (y,1)\bigr) \cdot (y,1)
        = \bigl((x \cdot \rho_u(y)) \cdot y, 1\bigr) \in L.
    \]
    Hence $(x \cdot \rho_u(y)) \cdot y \in L_X$, showing that condition (I'2) holds for $\widetilde{L}$.  
    Therefore, by \cref{Lemma:The equivalence of the ideal in semidirect product}, $\widetilde{L}$ is an ideal of $X \rtimes_\rho Y$.

    Since $X \sym_\rho Y$ is an L-subalgebra of $X \rtimes_\rho Y$, for any ideal $K$ of $X \rtimes_\rho Y$, the intersection 
    \[
        K \cap (X \sym_\rho Y)
    \]
    is an ideal of $X \sym_\rho Y$.  
    These two constructions are inverses of each other, establishing the claimed bijective correspondence.
\end{proof}

\section{Simple linear L-algebras and CKL-algebras}\label{Section: Simple linear L-algebras and CKL-algebras}

 In \cite[Lemma~4.3]{DietzelMenchonVendramin}, Dietzel, Menchón, and Vendramin have shown the following lemma for linear L-algebras.

\begin{lem}
\label{linear is KL}
Let $X$ be a linear algebra.  
For any $x, y, z \in X$ with $x \ge y > z$, one has
\[
    x \cdot y \;>\; x \cdot z .
\]

Moreover, every linear L-algebra is also a $KL$-algebra, i.e. $x\cdot y\geq y$ for all $x,y\in X$.
\end{lem}

\begin{pro}\label{Lem: invariant under ideals}
Let $X$ be a linear L-algebra with 
\[
X = \{x_0 > x_1 > \dots > x_{n-1}\},
\]
and suppose that $x_{i+1}$ is an invariant element of $X$.  
Let
\[
I = \uset x_i := \{x_j \mid j \le i\}.
\]
Then 
\[
x \cdot y = y \qquad \text{for all } x \in I,\; y \in X \setminus I.
\]
\end{pro}

\begin{proof}
We show that $x_{i+1}, \dots, x_{n-1}$ are invariant under the action of $I$.  
Proceed by induction.  Since $x_{i+1}$ is invariant, the base case holds.  
Assume that $x_k$ is invariant under $I$ for some $k > i+1$.  
Then by \cref{linear is KL},
\[
x_k = x \cdot x_k > x \cdot x_{k+1} \ge x_{k+1}
\]
for all $x \in I$.  
Thus $x \cdot x_{k+1} = x_{k+1}$ for all $x \in I$, proving the induction step.
\end{proof}

Using this result, we can give a characterization of ideals and prime ideals of a linear L-algebra.

\begin{thm}\label{Thm: linearL_ideals}
Let $X$ be a linear L-algebra with 
\[
X = \{x_0 > x_1 > \dots > x_{n-1}\},
\]
and let $I \subseteq X$.  
Then $I$ is an ideal of $X$ if and only if
\[
I = \uset x_i := \{x_j \mid j \le i\}
\]
for some $i \in \{0,\dots,n-1\}$, and moreover either $i = n-1$ or $x_{i+1}$ is an invariant element.
\end{thm}

\begin{proof}
Assume first that $I$ is an ideal of $X$, and let $x_i$ be the minimal element of $I$.  
By (I1), $I$ is upward closed, hence $I = \uset x_i$.

Suppose now that $i < n-1$ and choose any $y > x_{i+1}$.  
Then $y \in I$ while $x_{i+1} \notin I$, and by (I1) we must have $y \cdot x_{i+1} \notin I$.  
Thus $y \cdot x_{i+1} \le x_{i+1}$.  
On the other hand, by \cref{linear is KL},
\[
x_{i+1} \le y \cdot x_{i+1}.
\]
Hence $y \cdot x_{i+1} = x_{i+1}$ for all $y > x_{i+1}$, showing that $x_{i+1}$ is invariant.

Conversely, suppose that $I = \uset x_i$ and that $x_{i+1}$ is invariant.  
By \cref{rem: ideals of KL and CKL-alg}, it suffices to verify (I1) and (I3).
\begin{itemize}
    \item[(I1)]Let $x \in I$ and suppose $x \cdot y \in I$.  
By \cref{Lem: invariant under ideals}, if $y \notin I$ then $x \cdot y = y \notin I$, a contradiction.  
Therefore $y \in I$.
    \item[(I3)]If $x \in I$ and $y \notin I$, then by \cref{Lem: invariant under ideals},
\[
(x \cdot y)\cdot y = y \cdot y = 1 \in I.
\]

If $x \in I$ and $y \in I$, then since $X$ is a KL-algebra,
\[
(x \cdot y)\cdot y \ge y \ge x_i,
\]
hence $(x \cdot y)\cdot y \in I$.
\end{itemize}
Therefore, $I$ is an ideal of $X$.
\end{proof}

\begin{cor}\label{cor: prime ideals of linear algebras}
    Let $X=\{x_0>x_1>\dots>x_{n-1}\}$ be a linear L-algebra, then 
    \[
    \mathscr{I}(X)=\{\uset x_i\mid i=n-1\text{ or } x_{i+1} \text{ is invariant}\}
    \]
    and 
    $\Spec(X)=\mathscr{I}(X)\setminus\{X\}$.
\end{cor}
\begin{proof}
    The first claim is exactly what is proven in \cref{Thm: linearL_ideals}.
    
    Let $P$ be a proper ideal of $X$.
    We want to prove that it is prime.
    By the previous property, $P=\uset x_k$ for some $k\in\{0,\dots, n-2\}$ with $x_{k+1}$ invariant.
    Consider now any other ideal $I=\uset x_i$, then 
    \begin{align*}
    I\cdot P
    &=\{x_j\mid \langle x_j\rangle\cap I\subseteq P\}
    =\{x_j\mid \langle \uset x_j\rangle\cap \uset x_i\subseteq \uset x_k\}
    =\{x_j\mid \uset x_{\min(j,i)}\subseteq \uset x_k\}\\
    &=\{x_j\mid \min(j,i)\leq k\}=\begin{cases}
        X&\text{ if }i\leq k\\
        \uset x_k &\text{ if } i>k
    \end{cases}  
    =\begin{cases}
        X&\text{ if }i\leq k\\
        P &\text{ if } i>k
    \end{cases}  
    \end{align*}
    
    Therefore, $i\leq k$ and so $I\subseteq P$ or $i>k$ and $I\cdot P=P$.
    Hence, $P$ is a prime ideal.
\end{proof}

We now introduce a family of L-algebras $\{\mathbf{A}_n\}_{n \geq 1}$ and use the previous theorem to establish that each of these L-algebras is simple.

\begin{pro}\label{A_n is CKL}
    Let $n>1$ and $\mathbf{A}_n$ be the set $\{x_0,x_1,\dots, x_{n-1}\}$ 
    with multiplication defined as 
    $x_i\cdot x_j=x_{\max(j-i, 0)}$ for all $i,j\in \{0,\dots, n-1\}$.
    Then $\mathbf{A}_n$ is a simple linear CKL-algebra
    with $x_0>x_1>\dots>x_{n-1}$.
\end{pro}
\begin{proof}
    It is easy to check that
    \[
    \max(\max(k-i, 0)-\max(j-i, 0))=\max(\max(k-j, 0)-\max(i-j, 0))
    \]    
    for every $i,j,k\geq 0$, hence $(x_i\cdot x_j)\cdot (x_i\cdot x_k)=(x_j\cdot x_i)\cdot (x_j\cdot x_k)$ for every $x_i,x_j,x_k\in \mathbf{A}_n$.
    Moreover 
    $x_0\cdot x_i=x_i$ and
    $x_i\cdot x_0 =x_i\cdot x_i=x_0$ for all $x_i\in \mathbf{A}_n$
    and if $x_i\cdot x_j=x_j\cdot x_i=x_0$, then $i\leq j\leq i$, i.e. $x_i=x_j$.
    Therefore, $\mathbf{A}_n$ is an L-algebra with $1=x_0$ and $x_0<x_1<\cdots<x_{n-1}$.

    Moreover, 
    \[
    x_i\cdot(x_j\cdot x_k)=\begin{cases}
        x_{k-j-i} &\text{ if } k>i+j\\
        x_0 &\text{ otherwise }
    \end{cases}=x_j\cdot(x_i\cdot x_k),
    \text{ for every }x_i,x_j,x_k\in \mathbf{A}_n,\]
    so $\mathbf{A}_n$ is a CKL-algebra.

   Finally, to prove that it is simple, using \cref{Thm: linearL_ideals}, it is enough to show that there are no invariant elements apart from $x_0$ and $x_1$.
    Note that $x_{i-1}> x_i$ for every $i>1$ and $x_{i-1}\cdot x_i=x_1\neq x_i$, i.e. $x_i$ is not invariant for every $i>1$.
\end{proof}

\begin{lem}
\label{subalgebras of linear simple}
Let $n>1$ and
$X=\{x_0>x_1>\cdots x_{n-1}>x_n\}$ be a linear L-algebra.
Then $Y=X\setminus\{x_n\}$ is an L-subalgebra of $X$.
Moreover, $I$ is an ideal of $Y$ for every proper ideal $I\subset X$.
\end{lem}
\begin{proof}
    Let $I$ be a proper ideal of $X$, then $x_n\notin I$ and, more precisely, by \cref{Thm: linearL_ideals}, $I=\uset x_i$ for some 
    $i<n$ and $x_{i+1}$ is invariant in $X$.
    
    If $i=n-1$, then $I=Y$, which is an ideal of $Y$.

    Otherwise, $i<n-1$ and $x_{i+1}\in Y$ and $x_{i+1}$ is invariant also in $Y$. Therefore, by \cref{Thm: linearL_ideals}, $I$ is an ideal of $Y$.
\end{proof} 

The previous lemma allows us to use the inductive construction of linear algebras proved in \cite{DietzelMenchonVendramin}.
More precisely, \cite[Proposition 4.4]{DietzelMenchonVendramin} is the following.

\begin{pro}
\label{extension_linearLalg}
    Let $X=\{x_0>x_1>\cdots> x_{n-1}\}$ be a linear L-algebra
    and let $p\in X$ be the smallest invariant element of $X$.
    Consider now the poset $$L_{n+1}=\{x_0>x_1>\cdots x_{n-1}>x_n\}$$ and take $c\in L_{n+1}$ such that $p\cdot x_{n-1}>c$.
    Then there exists a unique L-algebra structure $X'$ on $L_{n+1}$ such that $X$ is an L-subalgebra of $X'$ and such that $p\cdot x_n=c$.
\end{pro}

\begin{thm}
\label{simple linear}
    Let $n>1$ and $X=\{x_0>x_1>\cdots >x_{n-1}\}$ be a linear L-algebra.
    If $X$ is simple, then $X$ is isomorphic to $\mathbf{A}_n$.
\end{thm}
\begin{proof}
    We prove the thesis by induction.
    For $n=2$, the claim is trivial.

    Let $n>1$ and $X=\{x_0>x_1>\cdots x_{n-1}>x_n\}$ be a linear simple L-algebra. Then, by \cref{subalgebras of linear simple}, $Y=\{x_0>x_1>\dots> x_{n-1}\}$ is a linear simple L-algebra too.
    Hence, by inductive hypothesis, $x_i\cdot x_j=x_{\max(j-i, 0)}$ for all $i,j\in \{0,\dots, n-1\}$.
    It remains to check that $x_i\cdot x_n=x_{n-i}$ for all $i\in \{0,\dots, n-1\}$.

    Notice that in $Y$ the smallest invariant element is $x_1$ and, since $x_n< x_1$, by \cref{linear is KL}, $x_1\cdot x_n<x_1\cdot x_{n-1}=x_{n-2}$.
    Moreover $x_1\cdot x_n$ cannot be $x_n$ otherwise, by \cref{linear is KL}, $x\cdot x_n=x_n$ for every $x\neq x_n$ i.e. $x_n$ is invariant, which is against the fact that $X$ is a linear simple L-algebra.
    Therefore $x_1\cdot x_n=x_{n-1}$ and, by \cref{extension_linearLalg}, there is a unique L-algebra structure on $X$ such that $Y$ is a L-subalgebra and $x_1\cdot x_n=x_{n-1}$, which is precisely $S_{n+1}$.
\end{proof}

In the remaining, we will extend \cref{simple linear} to a subclass of CKL-algebra, namely tail$^+$ CKL-algebras.

\begin{defn}\label{Def: tail+}
Let $X$ be an L-algebra, and let $z$ be a minimal element of $X$.  
The upset $\uset z$ of $z$
is called a \emph{tail} if 
it is a linear subset of $X$.

A finite L-algebra $X$ is called a \emph{tail$^+$ L-algebra} if
it has a tail or if it contains L-subalgebras
\[
Y \subseteq Y_0 \subseteq X
\]
such that:
\begin{enumerate}
    \item $Y$ has a tail;
    \item the set $Y_0 \setminus Y = \{z_0\}$ consists of a single element, which is the smallest element of $Y_0$;
    \item the complement $X \setminus Y$ is a linear poset.
\end{enumerate}

\end{defn}

In particular, any L-algebra $X$ whose Hasse diagram forms a directed tree 
is an L-algebra with $n$ tails, where $n$ denotes the number of leaves of the tree.

\begin{pro}\label{Lem: tail of minimal element in CKL is an ideal}
Let $X$ be a CKL-algebra with a minimal element $z \in X$. 
If the corresponding upset
\[
I := \uset z = \{\, x \in X \mid z \leq x \,\}
\]
is a tail. Then $I$
is an ideal of $X$.
\end{pro}

\begin{proof}
Assume that $I$ is a proper subset of $X$. 

We first show that $z \cdot y \notin I$ for all $y \notin I$.  Let $y\notin I$.
There exists a minimal element $x$ such that $y < x$ and $z < x$. 
Then, we have
\begin{align*}
    y \cdot z 
    &= (y \cdot x) \cdot (y \cdot z) \\
    &= (x \cdot y) \cdot (x \cdot z).
\end{align*}
Since $y \nleq z$, it follows that $x \cdot y \nleq x \cdot z$.  
Similarly, since 
\[
z \cdot y = (x \cdot z) \cdot (x \cdot y),
\]
we also have $x \cdot z \nleq x \cdot y$.  Hence, since $x \cdot z \in I$ and $I$ is linear, it follows that $x \cdot y \notin I$.
Moreover, since $X$ is a CKL-algebra, we obtain
\[
z \cdot (x \cdot y) = x \cdot (z \cdot y).
\]
Since $z \nleq x \cdot y$, it follows that $x \nleq z \cdot y$.  
Note that $z \cdot y \in \uset y$. Thus, 
\[
z \cdot y \in \uset y \setminus \uset x,
\]
which implies $z \cdot y \notin I$.

Next, we show that each $I$ satisfies property~(I1).  
Suppose $x,\, x \cdot y \in I$.  
Since $X$ is a CKL-algebra, we have
\[
x \cdot (z \cdot y) = z \cdot (x \cdot y) = 1.
\]
Hence $z \leq x \leq z \cdot y$, which means $z \cdot y \in I$.  
From the first part of this proof, it follows that $y \in I$.  
Therefore, $I$ is an ideal of $X$.
\end{proof}

By \cref{Lem: tail of minimal element in CKL is an ideal}, we can directly obtain the following result.

  \begin{exa}
  Let $X$ be a set $\{1, x, y, z\}$ with the following multiplication table:
\[
\begin{array}{c|cccc}
    & x & y & z & 1 \\
    \hline
    x & 1 & y & x & 1 \\
    y & x & 1 & z & 1 \\
    z & 1 & y & 1 & 1 \\
    1 & x & y & z & 1 \\
\end{array}
\]
It can be verified that $X$ is a CKL-algebra with the partial order $1 > y$ and $1 > x > z$.  
By \cref{Lem: tail of minimal element in CKL is an ideal}, we obtain two ideals of $X$:
\[
I_1 = \{1, x, z\} \quad \text{and} \quad I_2 = \{1, y\}.
\]

  \end{exa}

\begin{exa}
Let $X = \{1, x_1, x_2, x_3, x_4, x_5, x_6\}$ be a set equipped with the following multiplication table:
\[
\begin{array}{c|ccccccc}
    & x_1 & x_2 & x_3 & x_4 & x_5 & x_6 & 1 \\
    \hline
    x_1 & 1 & x_2 & x_3 & x_4 & x_5 & x_6 & 1 \\
    x_2 & 1 & 1 & x_3 & x_4 & x_5 & x_6 & 1 \\
    x_3 & 1 & x_2 & 1 & x_4 & x_3 & x_4 & 1 \\
    x_4 & 1 & 1 & x_3 & 1 & x_5 & x_3 & 1 \\
    x_5 & 1 & x_2 & 1 & x_4 & 1 & x_4 & 1 \\
    x_6 & 1 & 1 & 1 & 1 & x_3 & 1 & 1 \\
    1   & x_1 & x_2 & x_3 & x_4 & x_5 & x_6 & 1 \\
\end{array}
\]
It can be verified that $X$ is a CKL-algebra. 
The corresponding strict partial order on $X$ is represented by the following Hasse diagram:
\[
\begin{tikzcd}[row sep=2.2em, column sep=2.2em]
& 1 \arrow[d, no head] & \\
& x_1 \arrow[dl, no head] \arrow[dr, no head] & \\
x_2 \arrow[d, no head] & & x_3 \arrow[d] \arrow[ddll, no head]  \\
x_4 \arrow[d, no head] & & x_5 \\
x_6 & &
\end{tikzcd}
\]
By \cref{Lem: tail of minimal element in CKL is an ideal}, 
the tail $\uset x_5 = \{1, x_1, x_3, x_5\}$ is an ideal of $X$. 
In contrast, the upset $\uset x_6$
is not an ideal of $X$.
\end{exa}

  \begin{lem}
  \label{lem:CKL without minimum is a L-subalgebra}
      Let $X$ be a Glivenko algebra with the smallest element $0\in X$.
      Then $Y=X\setminus\{0\}$ is a CKL-subalgebra.
      Moreover, $I$ is an ideal of $Y$ if and only if $I$ is an ideal of $X$ or $I=Y$.
  \end{lem}
  \begin{proof}
      To prove that $Y$ is a CKL-subalgebra, it is enough to notice that if $x,y\in Y$, then $0<y\leq x\cdot y$, hence $x\cdot y\in Y$.

      Let now $I$ be an ideal of $Y$ that is not an ideal of $X$. Then there exists $x\in I$ such that the negation $x^*\in I$. We claim that $I=Y$.
      Let $y\in Y$, then $x^*\in I$ and  $x\cdot y\in Y$, but 
      \[
     x^*\cdot (x\cdot y)
 =(0\cdot x)\cdot (0\cdot y) 
     =1.
      \]
      So $x\cdot y\in I$, since $I$ is an ideal of $Y$.
      But now we have $x\in I$, $y\in Y$ such that $x\cdot y\in I$, thus $y\in I$.
  \end{proof}

  \begin{thm}\label{Thm: tail+ L-algebra}
  Let $n > 1$ and let $X$ be a tail$^+$ CKL-algebra of size $n$.  
If $X$ is simple, then $X$ is linear, hence it is isomorphic to $\mathbf{A}_n$.
  \end{thm}
 \begin{proof}
First, we start with the case when $X$ is a simple CKL-algebra with a tail. 
By \cref{Lem: tail of minimal element in CKL is an ideal}, $X$ has a unique minimal element.  
Hence, the partial order of $X$ is linear, and $X$ is isomorphic to $\mathbf{A}_n$.

Let $X$ be a tail$^+$ CKL-algebra.    
By \cref{lem:CKL without minimum is a L-subalgebra} and induction, $X$ is also linear and isomorphic to $\mathbf{A}_n$.
\end{proof}

Moreover, a CKL simple L-algebra cannot have more than one connected component in the Hasse diagram of $X\setminus\{1\}$ as the following proposition states.

\begin{pro}
    Let $X$ be a CKL-algebra and let $C$ be a connected component of the Hasse diagram of $X\setminus\{1\}$.
    Then $C\cup\{1\}$ is an ideal of $X$.
\end{pro}
\begin{proof}
    Thanks to \cref{rem: ideals of KL and CKL-alg}, 
    we only need to prove property (I1) of the definition of ideal.

    Let $x\in C\cup\{1\}$ and $y\in X$ such that $x\cdot y\in C\cup\{1\}$.
    Then either $x=1$ or $x\in C$. 
    
    If $x=1$, we have directly $y=x\cdot y\in C\cup\{1\}$.
    
    Otherwise, $x\in C$. Since $x\cdot y\in C\cup\{1\}$, we have two cases again: either $x\cdot y=1$ or $x\cdot y\in C$.
    If $x\cdot y=1$, then $x\leq y$. Hence, $y$ is connected to $x$ in the Hasse diagram. So $y=1$ or $y\in C$.
    Assume now that $x\cdot y\in C$.
    Using that $X$ is CKL, hence KL, we get that $y\leq x\cdot y$.
    Thus, $y$ is connected to $x$ in the Hasse diagram. So $y=1$ or $y\in C$.
    In any case, we proved that $y\in C\cup\{1\}$.
\end{proof} 

Given the previous proposition and based on computational results,
we have the following conjecture.

\begin{conjecture}
    Every finite simple CKL-algebra is linear.
\end{conjecture}

\section{Symmetric semidirect products and   Hilbert algebras}\label{Section: Symmetric semidirect products and   Hilbert algebras}

In this section, we mainly study the ideals, semidirect products of Hilbert algebras, and the structure of linear Hilbert algebras.

\begin{lem}
\label{Hilber_ideals}
    Let $X$ be a Hilbert algebra, and let $z\in X$.
    Then the upset $\uset z$ is the ideal $\langle z\rangle$ generated by $z$.
\end{lem}
\begin{proof}
       Note that $X$ is also a CKL-algebra, so $I\subseteq X$ is an ideal of $X$ if and only if $1\in I$ and 
    $I$ satisfies (I1).\begin{itemize}
        \item Clearly $1\in \uset z$.
        \item If $x,x\cdot y\in \uset z$, then 
        \[
        z\cdot y=1\cdot (z\cdot y)=(z\cdot x)\cdot (z\cdot y)=z\cdot (x\cdot y)=1,
        \]
        i.e. $y\in \uset z$.
    \end{itemize}
    Then $\uset z$ is an ideal of $X$. Thus, $\langle z\rangle\subseteq \uset z$. 

    For each $x\in \uset z$, we have $z\cdot x=1\in \langle z\rangle$. By condition (I1), we conclude that $x\in \langle z\rangle$. Therefore, $\langle z\rangle=\uset z$.
\end{proof}
\begin{pro}
    Let $X$ be a finite Hilbert algebra and let $I$ be an ideal of $X$.  
    Denote by $\min(I)$ the set of all minimal elements of $I$.  
    Then
    \[
        I \;=\; \bigcup_{z\in \min(I)} \uset{z}.
    \]

    Moreover, let $\min(X)=\{m_1,\ldots,m_n\}$ be the set of all minimal elements of $X$,  
    and for each $1\le i\le n$ define
    \[
         P_i \;=\; \bigcup_{m\in \min(X)\setminus\{m_i\}} \uset{m}.
    \]
   If $P$ is a proper ideal of $X$ such that the $P_i\subseteq P$ for some $1\leq i\leq n$ and $X\setminus P$ is linear,  
then $P$ is a prime ideal of $X$.
\end{pro}

\begin{proof}
    By \cref{Hilber_ideals}, we have $\langle z_i\rangle = \uset{z_i}\subseteq I$ for each 
    $z_i\in \min(I)$. Hence
    \[
        \bigcup_{z\in\min(I)} \uset{z} \;\subseteq\; I.
    \]

    Conversely, let $x\in I$.  
    Since $I$ is finite, it has minimal elements, and 
    every element of $I$ lies above some minimal element of $I$.  
    Thus, there exists $z_j\in\min(I)$ such that $z_j\le x$, i.e.\ $x\in\uset{z_j}$.  
    Therefore
    \[
        I \;\subseteq\; \bigcup_{z\in\min(I)} \uset{z}.
    \]

    Combining the two inclusions yields
    \[
        I = \bigcup_{z\in\min(I)} \uset{z},
    \]
    as claimed.

   Let $P$ be a proper ideal such that $P_i \subseteq P$ for some $1 \le i \le n$ and $X\setminus P$ is linear.  
Then 
\[
P = P_i \cup \uset z_i,
\qquad\text{where } m_i < z_i.
\]

Assume that $I \not\subseteq P$.  
Then there exists a minimal element $z \in I$ such that
\[
m_i \le z < z_i.
\]
Since $I \setminus P$ is linear, we have
\begin{align*}
I \cdot P
    &= \{\, x \in X \mid \uset x \cap I \subseteq P \,\}  \\
    &\subseteq \{\, x \in X \mid \uset x \cap \uset z \subseteq P \,\} \\
    &= \{\, x \in X \mid \uset x \subseteq P \,\} \\
    &\subseteq P.
\end{align*}
Thus $I \cdot P \subseteq P$ for every ideal $I$ with $I \not\subseteq P$.  
Therefore, $P$ is a prime ideal.
\end{proof}

Using \cref{Hilber_ideals}, it is now easy to show that there is only one simple Hilbert algebra.

\begin{pro}
\label{Hilbert_simple}
      Let $X$ be a Hilbert algebra. $X$ is simple if and only if $|X|\leq 2$.  
\end{pro}
\begin{proof}

    Let $z\in X$, then, by \cref{Hilber_ideals}, $\uset z=\{x\in X\mid z\leq x\}$ is an ideal of $X$.
    
    Assume now, by contradiction, that $X$ is simple and $|X|>2$, then there exist $z_1, z_2\in X\setminus \{1\}$ distinct elements.
    But $\uset z_1$ and $\uset z_2$ are non-trivial ideals, so $\uset z_1=X=\uset z_2$, which is a contradiction because we would have $z_1<z_2<z_1$.
\end{proof}

\begin{pro}
    Let $\LH_n = \{x_0, x_1, \dots, x_{n-1}\}$ with multiplication defined by 
    \[
        x_i \cdot x_j =
        \begin{cases}
            x_0 = 1, & \text{if } i \geq j, \\[4pt]
            x_j, & \text{if } i < j.
        \end{cases}
    \]
    Then $\LH_n$ is a linear Hilbert algebra.
\end{pro}

\begin{proof}
    For all $x_i, x_j, x_k \in \LH_n$, we have
    \[
        x_i \cdot (x_j \cdot x_k)
        =
        \begin{cases}
            x_k, & \text{if } j < k \text{ and } i < k, \\[4pt]
            1, & \text{otherwise.}
        \end{cases}
    \]
    On the other hand,
    \[
        (x_i \cdot x_j) \cdot (x_i \cdot x_k)
        =
        (x_j \cdot x_i) \cdot (x_j \cdot x_k)
        =
        \begin{cases}
            x_k, & \text{if } j < k \text{ and } i < k, \\[4pt]
            1, & \text{otherwise.}
        \end{cases}
    \]
    Hence, the defining identity of a Hilbert algebra,
    \[
        x_i \cdot (x_j \cdot x_k)
        = (x_i \cdot x_j) \cdot (x_i \cdot x_k)
        = (x_j \cdot x_i) \cdot (x_j \cdot x_k),
    \]
    holds for all $x_i, x_j, x_k \in \LH_n$. Therefore, $\LH_n$ is a Hilbert algebra.
\end{proof}

\begin{pro}
\label{Hilbert_linear}
 Let $n>1$ and $X=\{x_0>x_1>\cdots >x_{n-1}\}$ be a linear Hilbert algebra.
    Then $X$ is isomorphic to $\LH_n$.
    
\end{pro}
\begin{proof}
By the definition of L-algebra, $x_0=1$ is an invariant element.

Let now $j\in \N_{\geq 1}$. Since $X$ is Hilbert,
by \cref{Hilber_ideals} $\uset x_{j-1}$ is an ideal of $X$.
Moreover, since $X$ is also linear, by \cref{Thm: linearL_ideals}, $x_j$ is an invariant element. 
Therefore, we proved the thesis.
\end{proof}

\begin{cor}
    Let $X$ be a linear Hilbert algebra of size $n$, and let $I$ be an ideal of $X$. Then:
    \begin{enumerate}
        \item There exists an $\rho$ such that $I$ operates on $X/I$ via $\rho$ as Hilbert algebras, and 
        \[
            X \cong I \sym_{\rho} (X/I).
        \]
        
        \item Conversely, if there exists a $\rho$ such that $I$ operates on $Y$ via $\rho$ as Hilbert algebras and 
        \[
            X \cong I \sym_{\rho} Y,
        \]
        then $Y \cong X/I$.
    \end{enumerate}
\end{cor}

\begin{proof}
Let $I$ be a proper ideal of $X$.  
By \cref{Thm: linearL_ideals}, we have 
\[
    I = \uset x_i := \{\, x_j \mid j \leq i \,\},
\]
for some $i \in \{0, \dots, n-1\}$.

By \cref{Hilbert_linear} and \cref{Thm: linearL_ideals}, it follows that $X/I \cong \LH_{n-i+1}$.  
Define $\rho : X/I \to \End(I)$ by 
\[
    \rho_{[u]_I}(x) = 1, \quad \text{for all } x \in I, \, [u]_I \in X/I.
\]
Then 
\[
    I \sym_{\rho} (X/I) = (\{1\} \times X/I) \cup (I \times \{1\})
\]
is a linear Hilbert algebra of size $n$.  
By \cref{Hilbert_linear}, we obtain the isomorphism 
\[
    X \cong I \sym_{\rho} (X/I).
\]

Conversely, by \cref{Thm: ideals of symmetric semidirect product} and \cref{Thm: prime ideals of semidirect product of KL-algebras}, we have 
\[
    |\Spec(I \sym_\rho Y)| = |\Spec(I \rtimes_\rho Y)| = |\Spec(I)| + |\Spec(Y)|.
\]
By \cref{cor: prime ideals of linear algebras} and \cref{Hilbert_linear}, we know that 
\[
    |\LH_n| = |\Spec(\LH_n)| + 1.
\]
Hence, $|Y| = n - i + 1$.  
Since $Y$ is isomorphic to a Hilbert subalgebra of $X$, it follows that 
\[
    Y \cong \LH_{n - i + 1} \cong X/I.\qedhere
\]
\end{proof}

We now focus on Hilbert algebras that arise as extensions, via symmetric semidirect products, of the simple Hilbert algebra $\mathbf{A}_2=\{1>0\}$.

\begin{pro}\label{number of ideals of symmetric semidirect product of Hilbert algebras}
    Let $X$ be a Hilbert algebra and $\mathbf{A}_2$ be the simple Hilbert algebra such that $\mathbf{A}_2$ acts on $X$ via $\rho$ as Hilbert algebras. Let $I_0=\ker \rho_0$. Then   \[|\mathscr{I}(X\sym_\rho \mathbf{A}_2)|=|\mathscr{I}(X)|+|\mathscr{I}(X/I_0)|.\]
\end{pro}
\begin{proof}
Let $I_0=\ker \rho_0$. 
By \cref{Thm: prime ideals of semidirect product of KL-algebras}, $I_0$ is an $\rho$-ideal of $X$. Thus, we can induce $X/I_0$ to operate on $\mathbf{A}_2$ via $\rho^{I_0}$. Let $\bar{I}$ be an ideal of $X/I_0$.    For all $y\in X/I_0$, then \[\rho_0^{I_0}( \rho_0^{I_0}(y)\cdot y)= \rho_0^{I_0}(y)\cdot\rho_0^{I_0}(y)=1\]
    and 
    \[\rho_0^{I_0}( y\cdot\rho_0^{I_0}(y))= \rho_0^{I_0}(y)\cdot\rho_0^{I_0}(y)=1\]
    Since $\ker\rho_0^{I_0}=1$, then $\rho_0^{I_0}(y)\cdot y=y\cdot\rho_0^{I_0}(y)=1$, which means $\rho_0^{I_0}=\id_{X/I_0}.$
    Then, we have 
    \[(X/I_0)\rtimes_{\rho^{I_0}} \mathbf{A}_2=(X/I_0)\times \mathbf{A}_2.\]
    By \cref{Thm: ideals of symmetric semidirect product},
    \[
    |\mathscr{I}(X\sym_\rho \mathbf{A}_2)|=|\mathscr{I}(X)|+|\mathscr{I}(X/I_0)|.\qedhere
    \]
\end{proof}
\begin{cor}
     Let $X$ be a Hilbert algebra such that $\mathbf{A}_2$ operates on $X$ via $\rho$ as Hilbert algebras. Then 
     \[|\mathscr{I}(X\sym_\rho \mathbf{A}_2)|=|\mathscr{I}(X)|+1\] if and only if $X\sym_\rho \mathbf{A}_2=X\sqcup\{(1,0)\}$ is bounded with smallest element $(1,0)$.
\end{cor}
\begin{proof}
     $|\mathscr{I}(X/I_0)|=1$ if and only $\rho_0(x)$ for all $x\in X$. 
\end{proof}

\begin{exa}
Let $X=\{x_0 > x_1 > \cdots > x_{n-1}\}$ be a linear Hilbert algebra.  
For each $0 \le k < n$, define a map
\[
\rho^{(k)} : \mathbf{A}_2 \longrightarrow \End(X)
\]
by setting $\rho^{(k)}_1 = \id_X$ and defining $\rho^{(k)}_0 : X \to X$ as
\[
\rho^{(k)}_0(x_i) =
\begin{cases}
x_i, & \text{if } i > k,\\[3pt]
1,   & \text{if } i \le k .
\end{cases}
\]

It is straightforward to verify that for each $1 \le k < n$,  
$\mathbf{A}_2$ acts on $X$ via $\rho^{(k)}$ as Hilbert algebras.  
By \cref{Hilbert_linear}, the quotient $X / \ker \rho^{(k)}_0$ is isomorphic to the Hilbert chain $\LH_{n-k}$.  
Therefore, by \cref{number of ideals of symmetric semidirect product of Hilbert algebras}, we obtain
\[
    \lvert \mathscr{I}( X \sym_{\rho^{(k)}} \mathbf{A}_2 ) \rvert = 2n - k .
\]
\end{exa}

\section*{Credit authorship contribution statement}
All authors contributed equally to this work.

\section*{Acknowledgements}

The second author is supported by the National Key R$\&$D Program of China (No. 2024YFA1013803), by Shanghai Key Laboratory of PMMP (No. 22DZ2229014), and by China Scholarship Council ( No. 202406140151).

The first author is supported by Fonds Wetenschappelijk Onderzoek - Vlaanderen, via a PhD Fellowship for fundamental research, grant 11PIO24N. 

This work is partially supported by the project OZR3762 of Vrije Universiteit Brussel, 
and by Fonds Wetenschappelijk Onderzoek - Vlaanderen, via the Senior Research Project G004124N.

The authors thank Leandro Vendramin for insightful discussions (also about the title).
\bibliography{references}
\bibliographystyle{abbrv}

\end{document}